\documentclass[sort,hidelinks,AMA,Times1COL]{WileyNJDv5} 
\usepackage{dsfont}
\usepackage{algorithm}
\usepackage{algpseudocode}
\usepackage{accents}
\usepackage{pgfplots}
\usepackage{siunitx}
\usepgfplotslibrary{fillbetween}

\newcommand*{\vm}[1]{\boldsymbol{#1}}
\newcommand*{\trans}{^{\mathrm T}}
\newcommand*{\dd}{\mathrm{d}}
\newcommand*{\partialx}{\partial_{\vm x_i}}
\newcommand*{\partialu}{\partial_{\vm u_i}}
\newcommand{\neighs}{j \in \mathcal{N}_i}
\newcommand{\send}{j \in \mathcal{N}_i^\leftarrow}
\newcommand{\rec}{j \in \mathcal{N}_i^\rightarrow}

\newcommand{\agents}{i \in \mathcal{V}}
\newcommand*{\Ni}[1]{\vm{#1}_{\mathcal N_i^\leftarrow}}
\newcommand*{\Nj}[1]{\vm{#1}_{\mathcal N_j^\leftarrow}}

\newcommand*{\inds}{^{*}}
\newcommand*{\indqp}{^{q-1}}
\newcommand*{\indq}{^{q}}
\newcommand*{\indqk}{^{q_k}}
\newcommand*{\indqkp}{^{q_k-1}}
\newcommand*{\timetau}{\tau \in [0,\,T]}
\newcommand*{\timetausample}{\tau \in [0,\,\Delta t)}
\newcommand*{\stt}{\operatorname{s.t.}}
\newcommand\munderbar[1]{\underaccent{\bar}{#1}}
\newcommand*{\initstate}{\vm {\munderbar x}_k}

\algrenewcommand\ALG@beginalgorithmic{\relax}

\definecolor{myblue}{rgb}{0.00000,0.44700,0.74100}%
\definecolor{myblue2}{rgb}{0.3010, 0.7450, 0.9330}%
\definecolor{myred}{rgb}{0.85000,0.32500,0.09800}%
\definecolor{myyellow}{rgb}{0.9290 0.6940 0.1250}%
\definecolor {mypurple} {rgb} {0.4940 0.1840 0.5560}

\articletype{Research Article}%

\received{Date Month Year}
\revised{Date Month Year}
\accepted{Date Month Year}
\journal{Journal}
\volume{00}
\copyyear{2024}
\startpage{1}

\begin{document}

\title{Sensitivity-Based Distributed Model Predictive Control for Nonlinear Systems under Inexact Optimization}

\author[1]{Maximilian Pierer von Esch}

\author[1]{Andreas V\"olz}

\author[1]{Knut Graichen}

\authormark{Pierer v. Esch \textsc{et al.}}
\titlemark{Sensitivity-Based Distributed Model Predictive Control for Nonlinear Systems under Inexact Optimization}

\address[]{\orgdiv{Chair of Automatic Control}, \orgname{Friedrich-Alexander-Universität Erlangen-Nürnberg}, \orgaddress{Erlangen, \country{Germany}}}

\corres{Maximilian Pierer von Esch, Chair of Automatic Control, Friedrich-Alexander-Universität Erlangen-Nürnberg, 91058~Erlangen, Germany. \email{maximilian.v.pierer@fau.de}}

\fundingInfo{Deutsche Forschungsgemeinschaft (DFG, German Research Foundation)
	under project no.Gr 3870/6-1.}

\abstract[Abstract]{This paper presents a distributed model predictive control (DMPC) scheme for nonlinear continuous-time systems. The underlying distributed optimal control problem is cooperatively solved in parallel via a sensitivity-based algorithm. The algorithm is fully distributed in the sense that only one neighbor-to-neighbor communication step per iteration is necessary and that all computations are performed locally. Sufficient conditions are derived for the algorithm to converge towards the central solution. Based on this result, stability is shown for the suboptimal DMPC scheme under inexact minimization with the sensitivity-based algorithm and verified with numerical simulations. In particular, stability can be guaranteed with either a suitable stopping criterion or a fixed number of algorithm iterations in each MPC sampling step which allows for a real-time capable implementation.}

\jnlcitation{\cname{%
		\author{M. Pierer von Esch}, 
		\author{A. V\"olz}, 
		\author{K. Graichen}} 
	\ctitle{Sensitivity-Based Distributed Model Predictive Control for Nonlinear Systems under inexact optimization}}

\maketitle

\renewcommand\thefootnote{}
\footnotetext{\textbf{Abbreviations:} ADMM, alternating direction multiplier method; CLF, control Lyapunov function; DMPC, distributed model predictive control; MPC, model predictive control; OCP, optimal control problem}

\renewcommand\thefootnote{\fnsymbol{footnote}}
\setcounter{footnote}{1}

\section{Introduction}
\label{intro}
Model predictive control (MPC) has emerged as a powerful strategy for controlling both linear and nonlinear systems, demonstrating its versatility in handling constraints and optimizing cost functions.\cite{Rawlings,Mayne,Grune} However, confronted with large-scale systems, a centralized MPC approach for the entire system is often impractical or undesirable. This limitation may be due to the unavailability of a centralized control entity with sufficient computing resources, scalability concerns, or communication restrictions imposed by a given communication topology. These restrictions in combination with a shift towards networked and decentralized control architectures have led to the concept of distributed MPC (DMPC) in which the central MPC controller is replaced by local MPCs, which control the single subsystems of the global system.\cite{Negenborn,Camponogara,Christofides,AlGerwi} The subsystems together with the local MPCs then form a network of so-called agents. 

The development of a DMPC scheme heavily depends on the considered problem setting since various sorts of couplings between subsystems and communication topologies give rise to different requirements for the DMPC controller.\cite{Christofides}
For instance, the subsystems can be dynamically coupled in the sense that the dynamics of the individual subsystems depend on other subsystems' states and/or inputs. Such couplings typically arise from physically connected subsystems found in large-scale processing plants \cite{Scheu,Liu,Venkat2}, infrastructure networks such as smart grids\cite{Jiang,Braun,Venkat} and water distribution networks\cite{Leirens,trnka,hentzelt2} or coupled mechanical systems, for example, found in cooperative payload scenarios\cite{Wehbeh,Mulagaleti}. A majority of the DMPC literature considers linear discrete-time systems, in which the subsystems are assumed to be linearly coupled to the state and/or control of their neighbors.\cite{Kohler,Conte,Farina,Giselsson,Rostami} In recent years, however, also dynamically coupled nonlinear continuous-time systems have been the focus of research. \cite{Bestler,Burk,Burk2,Mirzaei,Dunbar}

In other settings, the global system is composed of a group of dynamically independent subsystems whose dynamics only depend on their own states and controls. Typically, this is the case in multi-agent networks in which the agents need to solve different cooperative tasks like formation control\cite{Zhihao,Zhao}, platooning\cite{LiuP}, synchronization and consensus\cite{Muller2}, flocking\cite{Lyu} or coverage problems\cite{Farina2}. These cooperative tasks are then often formulated in terms of coupled cost functions and/or constraints. The difficulty in this regard is that the control task usually differs from the classical setpoint stabilization scenario.   

Finally, the available communication topology has a major influence on the DMPC scheme since the communication topology in general does not need to correspond to the coupling structure which results in different design requirements. A large share of DMPC schemes considers either neighbor-to-neighbor communication where the communication and coupling topology are identical\cite{Farina,Burk2,Huber,Conte,Costantini3,Pierer}, star networks in which the agents communicate only with a central node \cite{Pierer2,Belgioioso} or system-wide communication which allows the agents to communicate with all other agents or a subset thereof\cite{Engelmann}. Consequently, different combinations of these communication structures are possible. Reducing communication is a vital aspect of DMPC, as previous studies have reported that communication can be responsible for a major part of the overall computation time \cite{Burk2,Burk4,Pierer,Rostami2}. Neighbor-to-neighbor communication is favorable in that regard as it lowers the system-wide communication load and provides a more flexible and scalable communication structure without a single point of failure.

In addition to different problem settings and communication requirements, the inherent characteristics of the DMPC controller can vary drastically in how neighboring subsystems are considered. Since a comprehensive and comparative overview of different DMPC structures is out of scope of this paper, we refer the reader to the survey papers \cite{Negenborn,Christofides,Muller}. However, one of the most promising approaches is distributed optimization, where the central MPC problem is solved iteratively and in parallel by each agent. The idea is that after enough algorithm iterations in each MPC step, the local solutions will converge to the centralized one leading to the same performance as the centralized MPC scheme which is known as cooperative DMPC.\cite{Muller} 

Distributed optimization algorithms require suitable decomposition schemes such that the central problem can be decomposed into smaller subproblems which are then solved locally and in parallel by the individual agents. For an introductory overview of different optimization methods used in DMPC, we refer to the surveys \cite{Stomberg,Braun2}. A majority of the distributed optimization methods in DMPC rely on a decomposition of the dual problem, where coupling constraints are taken into account by Lagrangian multipliers. Notable algorithms that fall into this category are dual decomposition\cite{Giselsson}, the alternating direction method of multipliers (ADMM) \cite{Boyd}, decentralized interior point methods \cite{Engelmann}, the augmented Lagrangian-based alternating direction inexact Newton method \cite{Houska} and its bi-level distributed variant\cite{Engelmann2}. A large share of DMPC schemes employ these dual decomposition-based optimization methods.\cite{Farokhi,Bestler,Burk2,Burk,Giselsson,Zhao,Rostami,Kogel,Kohler} However, the disadvantage of dual optimization methods such as ADMM lies in the fact that it achieves primal feasibility only in the limit, strong duality is required to hold, and slow convergence for highly coupled systems has been observed in practice.\cite{Pierer,Costantini3}  

Another class of decomposition methods concerns primal decomposition. These approaches are characterized by directly exchanging required quantities between neighboring subsystems rather than using Lagrangian multipliers to append the coupling constraints.\cite{Venkat3,Venkat,Scheu,Doan,Stewart,Stewart2} A comparison of primal and dual decomposition methods in DMPC is for example given in references \cite{Stomberg,Braun2,Paulen,Pflaum}. An advantage of primal approaches is that strong duality is not required to hold and that the iterates are primal feasible under certain conditions. However, it needs to be ensured that the distributed solution corresponds with the centralized one. Most notable, sensitivity-based approaches have recently emerged as viable distributed optimization algorithms for DMPC.\cite{Scheu,Scheu2,Alvarado,Huber,Huber2,Pierer} All these works follow the same ideas which can be traced back to the 1970s \cite{Cohen,Mesarovic}. In particular, sensitivity-based approaches have been utilized for automatic building control \cite{Darure,Huber2}, water distribution networks \cite{Alvarado}, or process control \cite{Scheu}. Furthermore, the contributions\cite{Pierer,Pierer3} show that the sensitivities can be calculated locally by each agent in a computationally efficient manner using optimal control theory resulting in a fully distributed algorithm with only neighbor-to-neighbor communication which overcomes the major drawback that the agents usually require knowledge of the overall system dynamics, additional communication steps or a central coordination step. In addition, convergence of the sensitivity-based algorithm in the context of DMPC has been investigated for the linear discrete-time case \cite{Scheu}, an algorithm variant considering the complete system dynamic at agent level \cite{Huber2} and most recently for nonlinear continuous-time systems in a fully distributed setting \cite{Pierer3} with an inexact optimization of the underlying OCP. This makes the sensitivity-based approach an interesting option for nonlinear DMPC.

Common approaches to ensure the stability of DMPC schemes utilizing distributed optimization algorithms, especially for dynamically coupled systems, are to use classical MPC terminal ingredients such as local terminal costs and set constraints which can be extended to the distributed setting.\cite{Muller3,Muller4, Conte,Costantini} Compared to the central MPC case, synthesis of the terminal costs and sets is non-trivial in the distributed case and usually requires optimization-based approaches \cite{Conte,Costantini}, although iterative algorithms exist \cite{Costantini2}.  However, terminal set constraints that are designed for stability requirements are often restrictive and unfavorable from a numerical viewpoint as they lead to an increased computational load compared to an MPC formulation without terminal constraints \cite{Graichen2} which extends to DMPC \cite{Bestler}. Another approach is relaxed dynamic programming, in which the terminal conditions are dropped and stability is guaranteed via a relaxed descent condition \cite{Grune2} which has found application in DMPC.\cite{Farokhi} An important aspect of a (D)MPC scheme is to ensure stability even after a finite number of optimization iterations\cite{Graichen2,Bestler}, often referred to as stability under inexact optimization \cite{Stomberg}, as this limits the time-consuming communication steps. Typically, a suitable stopping criterion or a lower bound on the required iterations is used to guarantee stability. \cite{Giselsson3,Bestler,Kohler,Belgioioso,Jane}

This contribution presents a fully distributed DMPC scheme for nonlinear continuous-time systems that are coupled in dynamics and/or cost functions via the states of their respective neighbors. The scheme is fully distributed in the sense that only neighbor-to-neighbor communication and local computations are required. The centralized MPC problem is formulated without terminal set constraints which allows for a computationally efficient solution.\cite{Graichen,Graichen2} Stability is ensured via terminal cost functions which act as a control-Lyaponuv function (CLF) and are synthesized offline via an optimization-based approach such that a separable structure is obtained. The DMPC scheme is realized via distributed optimization with a sensitivity-based primal decomposition approach. In order to limit the communication between subsystems, the algorithm is terminated either after a suitable stopping criterion or a fixed number of iterations are reached. Prematurely stopping the algorithm leads to a suboptimal control solution and subsequently a mismatch between the optimal centralized and distributed solution. Although related schemes have been investigated in \cite{Bestler,Giselsson3,Kohler}, the approach presented in this paper requires far less assumptions and considers nonlinear systems. In particular, by establishing linear convergence of the sensitivity-based algorithm, exponential stability and incremental improvement of the DMPC scheme are shown if either the stopping criterion is sufficiently small or if a certain number of fixed algorithm iterations are performed in each MPC step.

The paper is structured as follows: In Section \ref{sec:prob_state} the problem statement and system class are introduced for which the central optimal control problem (OCP) is stated. Subsequently, the optimal MPC and DMPC control strategy is reviewed and it is shown how separable terminal costs in the continuous-time nonlinear case can be synthesized. Following the central considerations, the distributed solution via the sensitivity-based approach is discussed in Section \ref{sec:dist_sol}. In particular, the notion of sensitivities is derived via classical optimal control theory, and the sensitivity-based DMPC scheme is presented. Sufficient conditions are derived such that the algorithm converges towards the central optimal solution. The stability of the closed-loop system under the DMPC control law is analyzed in Section \ref{sec:DMPC_stab}. The algorithm is evaluated via numerical simulations and compared to ADMM in Section \ref{sec:num_eval} before the paper is summarized in Section \ref{sec:concl}.

Several conventions are used throughout this text. The standard
2-norm $ \| \vm q \| := \| \boldsymbol q \|_2 = (|q_1|^2 + ...+
|q_n|^2)^{\frac{1}{2}} $ is used for vectors $ \boldsymbol q \in \mathds{R}^n$, while for
time functions $\boldsymbol q: [0,T] \rightarrow
\mathds{R}^n$, the vector-valued $L_\infty$-norm with 
$\|\vm q\|_{L_\infty} = \max_{t \in [0,T]} \| \vm q(t)\|$ is utilized along with
the corresponding function space $L_\infty(0,T;P)$ such that 
$\vm q \in L_\infty(0,T;P)$ implies $\|\vm q\|_{L_\infty}<\infty$ on $ t\in [0,T]$ and
${\vm q(t)\in P\subseteq\mathds{R}^n,\, t\in[0,T]}$. An
$r$-neighborhood of a point $ \vm v_0 \in \mathds{R}^v$ is denoted
as $\mathcal{B}(\vm v_0,r):=\{\vm v \in \mathds{R}^v \,|\, \|\vm v -
\vm v_0\|\leq r\}$, while an $r$-neighborhood to a set $\mathcal{S}
\subset \mathds{R}^v$ is defined as $\mathcal{S}^r := \bigcup_{\vm
	s_0 \in \mathcal{S}}\mathcal{B}(\vm s_0,r)$. The stacking and
reordering of individual vectors $ \vm v_i \in \mathds{R}^{v_i},\, i
\in \mathcal{V}$ from a set $ \mathcal{V}$ is defined as $ [\vm 
v_i]_{\agents}$. The partial derivative of a function $f(\vm x,\vm y)$ 
w.r.t.\ to one of its arguments $\vm x$ is denoted as 
$\partial_{\vm x}\, f(\vm x, \vm y)$. For given iterates $(\vm x^k, \vm
y^k)$ at step $k$ of an arbitrary algorithm, the short-hand notation
$\partial_{\vm x}\, f^k = \partial_{\vm x} \,f( \vm x, \vm y)\big|_{\vm x = \vm x^k, \vm y = \vm y^k}$ is used when applicable. The argument of the time functions is omitted in the paper when it is convenient. Finally, system variables are underlined (e.g. $\munderbar x $)  to distinguish them from MPC-internal variables.

%

\section{Problem statement and (D)MPC strategy}
\label{sec:prob_state}
The structure of multi-agent systems is conveniently described by a graph  $\mathcal{G} = (\mathcal{V},\mathcal{E})$ in which the vertices $\mathcal{V} = \{1,\hdots,N\}$ represent single dynamic subsystems referred to as agents and the edges $\mathcal{E} \subset \mathcal{V} \times \mathcal{V}$ reflect a coupling between two agents. In this paper, the dynamics of an agent $\agents $ with states $ \vm{\munderbar x}_i \in \mathds{R}^{n_i}$ and controls $  \vm{\munderbar u}_i \in \mathds{R}^{m_i}$ are described by the following nonlinear neighbor-affine system
\begin{align}\label{eq:agent_dyn}
	\vm{\dot{\munderbar x}}_i &= \vm f_{ii}(\vm{\munderbar x}_i, \vm{\munderbar u}_i)+ \sum_{\send} \vm f_{ij}(\vm{\munderbar x}_i, \vm{\munderbar x}_j) =:\vm f_i(\vm{\munderbar x}_i, \vm{\munderbar u}_i, \Ni{ \munderbar x})\,, \quad \vm{\munderbar x}_i(0) = \vm{\munderbar x}_{i,0}\,.
\end{align}
The coupling to the neighbors is given by the states $\vm{\munderbar x}_j \in \mathds{R}^{n_j}$ using the stacked notation $ \Ni{\munderbar x}:=[\vm{\munderbar x}_j]_{\send} \in \mathds{R}^{p_i}$ with $ p_i = \sum_{\send} n_j$ and $ p = \sum_{\agents} p_i$ where the set of sending neighbors $\mathcal{N}^\leftarrow_i = \{j \in \mathcal{V}\,:\,(j,i) \in \mathcal{E},\, j\neq i\}$ describes all the agents influencing agent $\agents$ while the set of receiving neighbors $\mathcal{N}_i^\rightarrow = \{j \in \mathcal{V}\,:\,(i,j) \in \mathcal{E},\, i\neq j \}$ refers to the agents being influenced by agent $\agents$. The union of both sets defines the neighborhood $\mathcal{N}_i = \mathcal{N}^\leftarrow_i \cup \mathcal{N}_i^\rightarrow$. In addition, it is allowed that the agents are able to communicate bi-directionally with all neighbors $\neighs$. Neighbor-affinity means that in addition to the local dynamic functions $\vm f_{ii}: \mathds{R}^{n_i} \times \mathds{R}^{m_i} \rightarrow \mathds{R}^{n_i}$, the neighbors' dynamic functions $\vm f_{ij}: \mathds{R}^{n_i} \times \mathds{R}^{n_j} \rightarrow \mathds{R}^{n_i}$ enter additively into the dynamics~\eqref{eq:agent_dyn} and only depend on exactly one other state $\vm{ \munderbar x}_j$, $\send$ .\cite{Burk2} The  individual controls $ \vm{\munderbar u}_i$ are constrained to the compact and convex constraint sets $\vm{\munderbar u}_i(t) \in \mathbb{U}_i \subset \mathds{R}^{m_i}$, $t\geq0$ which contain the origin $\vm 0 \in \mathds{R}^{m_i}$ in their interior.
The coupled subsystems \eqref{eq:agent_dyn} can equivalently be written in a centralized form 
\begin{align}\label{eq:centr_dyn}
	\vm{\dot{\munderbar x}} = \vm f(\vm{\munderbar x}, \vm{\munderbar u})\,, \quad  \vm{\munderbar x}(0) = \vm{\munderbar x}_0
\end{align}
with the central dynamics $ \vm f= [\vm f_i]_{\agents}$ as well as stacked state $\vm{\munderbar x} =  [\vm{\munderbar x}_i]_{\agents} \in \mathds{R}^n$ and control $\vm{\munderbar u} = [\vm{\munderbar u}_i]_{\agents} \in \mathds{R}^m$ vectors of dimension $n = \sum_{\agents} n_i$ and $m = \sum_{\agents} m_i$, respectively. The stacked initial state is given by $\vm{\munderbar x}_0 = [\vm{\munderbar x}_{i,0}]_{\agents} \in \mathds{R}^n$. The local input constraints are summarized as $ \vm{\munderbar u}(t) \in \mathbb{U} \subset \mathds{R}^m$ with the set $\mathbb{U}$ defined as the Cartesian product $ \mathbb{U} = \prod_{\agents} \mathbb{U}_i$. Furthermore, it is assumed that the local dynamics \eqref{eq:agent_dyn} and consequently the central dynamics \eqref{eq:centr_dyn} as well as cost functions \eqref{eq:def_central_cost} are at least twice continuously differentiable w.r.t. their arguments. In addition, system \eqref{eq:centr_dyn} must yield a bounded solution for any initial condition $\vm{\munderbar x}_0 \in \mathds{R}^n$ and input $\vm{\munderbar u}(t) \in \mathbb{U}$, $t \in [0, T]$ for some $0<T<\infty$. Without loss of generality, the control tasks consist in controlling each subsystem~\eqref{eq:agent_dyn} to the origin, i.e. $ \vm f_i(\vm 0, \vm 0, \vm 0) = \vm 0$ or equivalently  $\vm f(\vm 0, \vm 0) =  \vm 0 $ in the centralized form \eqref{eq:centr_dyn}. 
\subsection{Central optimal control problem and MPC stability}
\label{subsec:central_OCP}
In this section, well known centralized MPC results are summarized which will form the basis for the following investigations in the distributed case. \cite{Rawlings,Graichen2,Graichen3,Bestler} In particular, the central MPC problem relies on the repeated online solution of the following OCP at each sampling point $t_k = k \Delta t$, $ k\in \mathds{N}_0$
\begin{subequations}\label{eq:central_ocp}
	\begin{align}
		\min_{ \vm u} \quad&
		J(\vm u;\vm{\munderbar x}_k)= V(\vm x(T)) + \int_0^T l(\vm x(\tau), \vm u(\tau))\,\dd \tau
		\label{eq:central_costFunction}	\\
		~\stt \quad&	\vm{\dot{x}} = \vm f(\vm x, \vm u)\,, \quad\vm x(0) =\vm{\munderbar x}_{k}	\label{eq:central_ocp_dyn}\\
		&\vm u(\tau) \in \mathbb{U}\,,\quad \timetau \label{eq:central_input_constr}
	\end{align}
\end{subequations}
where $\vm{\munderbar x}(t_k) = \vm{\munderbar x}_k = [ \vm{\munderbar x}_{i,k}]_{\agents} $ is the state of system \eqref{eq:centr_dyn} at time $t = t_k$  and  $\Delta t>0$ is the sampling time. The cost function \eqref{eq:central_costFunction} with horizon length $T>0$ consists of the separable terminal and integral costs
\begin{align} \label{eq:def_central_cost}
	V(\vm x(T)) = \sum_{\agents} V_i(\vm x_i(T))\,, \quad 	l(\vm x, \vm u) = \sum_{\agents} l_i(\vm x_i, \vm u_i, \Ni{x})
\end{align}
with $V_i: \mathds{R}^{n_i} \rightarrow \mathds{R}^+_0 $ and $ l_i: \mathds{R}^{n_i} \times \mathds{R}^{m_i} \times \prod_{\neighs}\mathds{R}^{n_j} \rightarrow \mathds{R}^+_0$. Similar to \eqref{eq:agent_dyn}, the coupled integral cost function \eqref{eq:def_central_cost} is given in the neighbor-affine form
\begin{equation} \label{eq:agent_cost}
	l_i(\vm x_i, \vm u_i, \Ni{x}):= l_{ii}(\vm x_i, \vm u_i) + \sum_{\send} l_{ij}(\vm x_i, \vm x_j)
\end{equation}
with the local part $ l_{ii}: \mathds{R}^{n_i} \times  \mathds{R}^{m_i} \rightarrow \mathds{R}_0^+$ and the coupled part $l_{ij}: \mathds{R}^{n_i} \times  \mathds{R}^{n_j} \rightarrow \mathds{R}_0$. The terminal and integral cost additionally satisfy the quadratic bounds
\begin{align} \label{eq:bound_costs}
	m_l ( \|\vm x\|^2 + \| \vm u\|^2) \leq l(\vm x,\vm u) \leq M_l(\|\vm x\|^2 + \| \vm u\|^2)\,, \quad 
	m_V \|\vm x\|^2 \leq V(\vm x) \leq M_V \|\vm x\|^2
\end{align}
for some constants $M_l \geq m_l>0$ and $M_V \geq m_V>0$. The admissible input space to OCP \eqref{eq:central_ocp} then follows as $\mathcal{U}=L_{\infty}(0,T;\mathbb{U})$. Furthermore, there exists a non-empty and open set $\Gamma \subset \mathds{R}^n$ with $ \vm 0 \in \Gamma$ such that for all $\vm{\munderbar x}_{k} \in \Gamma $, a minimizing and unique solution $(\vm x\inds(\tau; \vm{\munderbar x}_k)$, $\vm u\inds(\tau; \vm{\munderbar x}_k))$, $\timetau$ of \eqref{eq:central_ocp} exists. The existence of an optimal solution is not too restrictive as no terminal constraints are considered and all functions are assumed to be continuously differentiable. Throughout this paper, the central necessary optimality conditions to OCP \eqref{eq:central_ocp} are needed. To this end, define the central Hamiltonian as
	\begin{equation} \label{eq:central_Hamiltonian}
		H(\vm x, \vm u, \vm \lambda) = l(\vm x, \vm u) + \vm \lambda\trans \vm f(\vm x, \vm u)
	\end{equation}
	with the adjoint state $\vm \lambda \in \mathds{R}^n$. Then, the first-order optimality conditions require that there exist optimal adjoint states $\vm \lambda\inds(\tau; \vm{\munderbar x}_k)$, $\timetau$ such that the canonical boundary value problem
	\begin{subequations}
		\begin{align}
			\vm{\dot x}\inds(\tau; \initstate) &= \vm f(\vm x\inds(\tau; \initstate), \vm u\inds(\tau; \initstate))\,, &&\quad  \vm x\inds(0; \initstate) = \initstate \label{eq:opt_statedyn} \\
			\vm{ \dot \lambda}\inds(\tau; \initstate) &= -\partial_{\vm x} H(\vm x\inds(\tau; \vm{\munderbar x}_k), \vm u\inds(\tau; \vm{\munderbar x}_k), \vm \lambda\inds(\tau; \vm{\munderbar x}_k))=: \vm G(\vm x\inds, \vm u\inds, \vm \lambda\inds)\,, &&\quad \vm \lambda\inds(T; \initstate) = \partial_{\vm x} V(\vm x\inds(T; \initstate)) \label{eq:opt_adjointdyn}
		\end{align}
		is satisfied and that $ \vm u\inds(\tau;\initstate)$ minimizes the Hamiltonian \eqref{eq:central_Hamiltonian}
		\begin{align} \label{eq:Pontraygin}
			\min_{\vm u \in \mathbb{U}} H(\vm x\inds(\tau; \vm{\munderbar x}_k), \vm u, \vm \lambda\inds(\tau; \vm{\munderbar x}_k))\,, \quad \timetau\,.
		\end{align}
	\end{subequations}
The corresponding optimal (minimal) cost is denoted as
\begin{align}\label{eq:opt_cost}
	J^*(\initstate) := J(\vm u^*(\cdot, \initstate); \initstate) = \sum_{\agents} J_i(\vm u_i\inds(\cdot; \initstate); \initstate)\,,
\end{align}
where $J_i(\vm u_i\inds(\cdot; \initstate); \initstate) =: J_i\inds(\initstate)$ are the optimal agent costs. 
MPC strategies usually assume that this optimal solution to OCP~\eqref{eq:central_ocp} is exactly known at each sampling point $t_k$. Then, the first part of the optimal control trajectory on the sampling interval $[t_k,t_{k+1})$ is applied to the centralized system \eqref{eq:centr_dyn} which can be interpreted as a nonlinear control law of the form
\begin{align}\label{eq:MPC_control_law}
	\vm{\munderbar u}(t_k +\tau) = \vm u^*(\tau; \initstate) =: \vm \kappa\inds(\vm x^*(\tau;\initstate ); \initstate) \,, \quad \timetausample
\end{align}
with sampling time $\Delta t <T$ and  $\vm \kappa\inds( \vm 0;\initstate) =  \vm 0$.
In the the next MPC time step $t_{k+1}$, the process of solving \eqref{eq:central_ocp} is repeated again with $\vm x(0) = \vm{ \munderbar x}_{k+1}$ that (in the nominal case) is given by $\vm{\munderbar x}_{k+1} = \vm x^*(\Delta t; \initstate)$. 
Since the central MPC problem is formulated without terminal constraints, it is often assumed that the terminal cost $V(\vm x)$ in \eqref{eq:bound_costs} represents a local CLF on an invariant set $\Omega_{\beta} \subset \Gamma$ containing the origin at its center.\cite{Graichen2,Jadbabaie}
\begin{assumption} \label{ass:CLF}
	There exists a feedback law $\vm u = \vm r(\vm x) \in \mathbb{U}$ and a non-empty compact set $\Omega_\beta :=\{ \vm x\in \Gamma: V(\vm x)\leq \beta\} \subset \Gamma $ containing the origin at its center such that for all $\vm x \in \Omega_\beta$ the CLF inequality
	\begin{equation} \label{eq:CLF_inequality}
		\dot{V}(\vm x, \vm r(\vm x)) + l(\vm x, \vm r(\vm x))\leq 0
	\end{equation}
	with $ \dot{V}(\vm x, \vm r(\vm x)) = \frac{\partial V}{\partial \vm x}\vm f(\vm x,  \vm r(\vm x))$ is satisfied.
\end{assumption}
A classical approach in MPC is to choose $V(\vm x)$ as the quadratic function $V(\vm x) = \vm x\trans \vm P\vm x$ where $\vm P\succ0$ follows from the solution of a Lyapunov or Riccati equation, given that the system \eqref{eq:central_ocp} is stabilizable around the origin. This results in a linear feedback law $ \vm r(\vm x) = \vm K \vm x$ which stabilizes the nonlinear system \eqref{eq:centr_dyn} on the (possibly small) invariant set $\Omega_{\beta}$.\cite{Chen,Parisini,Michalska} However, due to the structural constraint~\eqref{eq:def_central_cost} concerning the terminal cost $ V(\vm x)$ and the requirement that only direct neighbors are allowed to communicate, the design is more involved in the distributed setting.\cite{Conte,Costantini} Therefore, a simple optimization-based approach to design $V(\vm x)$ which considers the separability constraint \eqref{eq:def_central_cost} and to synthesize a structured control law, i.e. $\vm r(\vm x)= [\vm r_i( \vm x_i,\vm x_{\mathcal{N}_i})]_{\agents}$ with $ \vm x_{\mathcal{N}_i}:=[\vm x_j]_{\neighs}$,  will be discussed in Section~\ref{subsec:DMPC_stab_opt}. 
Throughout this paper, the following compact level set of the optimal cost 
\begin{align} \label{eq:reg_of_attr}
	\Gamma_\alpha = \{\vm x \in \Gamma : J^*(\vm x) \leq \alpha\}\,, \quad  \alpha := \beta\big(1+\frac{m_l}{m_V} T\big)\,,
\end{align}
which characterizes the domain of attraction of the MPC scheme, will be needed. Based on the preceding assumptions, the following stability results for the MPC scheme without terminal constraints under the centralized control law \eqref{eq:MPC_control_law} can be shown.\cite{Graichen2,Limon}
\begin{theorem} \label{th:centralMPC}
	Suppose that the Assumption \ref{ass:CLF} is satisfied. Then, for $\vm{ \munderbar x}_0 \in \Gamma_\alpha$ it holds that $\initstate \in \Gamma_\alpha$ for all MPC steps and $\vm x\inds(T; \initstate ) \in \Omega_\beta$ with $\Omega_\beta \subset \Gamma_\alpha$. Furthermore, the optimal cost in the next MPC step decreases according to
	\begin{equation} \label{eq:th1}
		J\inds(\vm x\inds(\Delta t; \initstate))\leq J\inds (\initstate) - \int_0^{\Delta t} l \big(\vm x\inds(\tau; \initstate), \vm u\inds(\tau; \initstate)\big)\,\dd \tau 
	\end{equation}
	for all $ \initstate \in \Gamma_\alpha$ and the origin of the system \eqref{eq:centr_dyn} under the centralized MPC control law \eqref{eq:MPC_control_law} is asymptotically stable in the sense that the closed-loop trajectories satisfy $\lim_{t \rightarrow \infty} \|\vm{ \munderbar x}(t)\| = 0$. 
\end{theorem}
Theorem \ref{th:centralMPC} and the corresponding proofs can be found in references\cite{Graichen2,Graichen3,Limon}. Asymptotic stability of the closed-loop trajectories follows from \eqref{eq:th1} using  Barbalat's Lemma and can be strengthened to exponential stability if additionally the optimal cost \eqref{eq:opt_cost} is continuously differentiable. Theorem \ref{th:centralMPC} will serve as the basis for further stability considerations in the distributed case. 

\subsection{Terminal cost design for distributed MPC}
\label{subsec:DMPC_stab_opt}
In this section it is discussed how the central MPC stability considerations of the previous section can be transferred to the distributed case and in particular how the separable terminal costs $V(\vm x) = \sum_{\agents} V_i(\vm x_i)$ and structured terminal controllers $\vm r(\vm x)= [\vm r_i( \vm x_i,\vm x_{\mathcal{N}_i})]_{\agents}$ can be synthesized in the considered nonlinear continuous-time setting such that stability is ensured. 
In cooperative distributed MPC, the central OCP \eqref{eq:central_ocp} is decomposed into smaller subproblems which are solved in parallel by the agents via a distributed optimization algorithm. After an optimal input has been found, each agent locally applies the controls
\begin{equation}\label{eq:dist_opt_controllaw}
	\vm{\munderbar u}_i(t_k + \tau ) = \vm u_i\inds(\tau; \initstate)\,, \quad \timetausample\,, \quad \agents
\end{equation}
to the subsystems \eqref{eq:agent_dyn}. Compared to the central case, the terminal feedback law $\vm r(\vm x)$ is not allowed to consider the complete system state in the distributed setting, i.e. $ \vm r_i(\vm x)$ for all $\agents$, since this violates the requirement of only neighbor-to-neighbor communication. Rather $\vm r_i$ can only involve neighboring states $ \vm x_j,\, \neighs$ which can be communicated, i.e. $\vm r_i( \vm x_i,\vm x_{\mathcal{N}_i})$. However, classical central MPC stability proofs, e.g. the works \cite{Graichen2,Chen,Graichen3,Jadbabaie}, rely on a full state feedback law. In addition, the required separability in \eqref{eq:def_central_cost} imposes an additional structural constraint on the CLFs. These two additional requirements need to be accounted for when transferring Theorem \ref{th:centralMPC} to the DMPC setting.
Therefore, it is necessary to replace Assumption \ref{ass:CLF} with the following strengthened assumption.
\begin{assumption}\label{ass:dist_CLF}
	There exists a structured feedback law $\vm u = [\vm r_i( \vm x_i,\vm x_{\mathcal{N}_i})]_{\agents} \in \mathbb{U}$ and a structured CLF, i.e. $V( \vm x) = \sum_{\agents} V_i(\vm x_i)$ such that the CLF inequality \eqref{eq:CLF_inequality} holds for all $[\vm x_i]_{\agents} \in \Omega_{\beta} = \{ \vm x \in \Gamma : V(\vm x) \leq \beta\}$. 
\end{assumption}
The assumption concerning the existence of such a structured feedback controller is typical in the DMPC setting \cite{Conte,Costantini,Costantini2,Kohler} and has recently been shown to be satisfied by a broad class of systems \cite{Zhang}. Furthermore, this assumption is not as restrictive as requiring local stabilizing controllers, i.e. $\vm u = [\vm r_i(\vm x_i)]_{\agents}$ \cite{Dunbar}. The derivation of general conditions for the existence of such (linear) structured feedback controllers is beyond the scope of this paper and poses a difficult problem in itself. Related literature includes the reference\cite{Siljak}, where conditions on the decentralized stabilizability of the linearized dynamics \eqref{eq:centr_dyn} are discussed, or the work\cite{Xu} where conditions on the coupling strength between agents are derived such that a structured controller exists. Furthermore, algorithms for designing structured feedback controllers for continuous-time systems can be found e.g. in the contributions \cite{Deroo,Maartensson}.
The following corollary reveals that Theorem \ref{th:centralMPC} holds in the same manner for the optimal distributed setting under the strengthened Assumption \ref{ass:dist_CLF}. 
\begin{corollary} \label{th:dist_opt_MPC}
	Suppose that Assumption \ref{ass:dist_CLF} holds. Then, the origin of the closed-loop system under the DMPC control law \eqref{eq:dist_opt_controllaw} is asymptotically stable for all $\vm {\munderbar x}_0 \in \Gamma_{\alpha}$ and the optimal cost \eqref{eq:opt_cost} decays asymptotically.
\end{corollary}
The proof of Corollary \ref{th:dist_opt_MPC} is analogous to the proof of Theorem \ref{th:centralMPC} as Assumption \ref{ass:dist_CLF} represents a special case of Assumption~\ref{ass:CLF}.  
It shows that the classical MPC result of Theorem \ref{th:centralMPC} extends to the distributed setting if it is additionally assumed that the feedback control $\vm r(\vm x)$ and terminal cost $ V(\vm x)$ are distributed as well. Note that this  result is independent of the underlying distributed optimization algorithm. The case in which only a suboptimal solution with a limited amount of iterations is found by the considered sensitivity-based distributed optimal control algorithm will be investigated in Section \ref{sec:DMPC_stab}. 
In the following, a simple optimization-based approach is proposed to obtain a structured terminal cost $ V(\vm x)$ and controller $ \vm u =\vm r(\vm x)$ offline such that Assumption \ref{ass:dist_CLF} is satisfied. For the remainder of this section, we will consider the specific quadratic terminal and integral cost functions 
\begin{equation} \label{eq:quadratic_costs}
	V(\vm x) = \sum_{\agents}\vm x_i\trans \vm P_i \vm x_i\,, \quad 	l(\vm x, \vm u) = \sum_{\agents} \vm x_i\trans \vm Q_i \vm x_i + \vm u_i\trans \vm R_i \vm u_i
\end{equation}
with weighting matrices $\vm P_i,\,\vm Q_i\succ0$ and $ \vm R_i\succ0$, $\agents$.
The main difficulty and difference to classical MPC are that the quadratic terminal cost $V(\vm x) = \vm x\trans \vm P\vm x$ and linear feedback law  $ \vm r(\vm x) = \vm K \vm x$ cannot be obtained by solving a Lyapunov or Riccati equation.\cite{Conte,Costantini} This is due to the fact that in general the solution matrix $\vm P$ and feedback controller $\vm K$ will not exhibit the separable structure as required by Assumption \ref{ass:dist_CLF}.
Furthermore, consider the linearized system of \eqref{eq:centr_dyn} at the origin which can be derived as 
\begin{equation} \label{eq:lin_system}
	\vm{\dot x} = \vm A \vm x + \vm B \vm u\,, \quad  \vm x(0) = \vm {\munderbar x}_0\,,
\end{equation}
where $ \vm A = \partial_{\vm x}\, \vm f(\vm 0,\vm 0)$ and $ \vm B:= \partial_{\vm u}\, \vm f(\vm 0,\vm 0)$. Note that similar to \eqref{eq:agent_dyn}, the linearized system exhibits a sparse structure with $\partial_{\vm x_j}\, \vm f_i(\vm 0, \vm 0, \vm 0)=:\vm A_{ij} \neq \vm 0$ only if $\send$ and $\vm B = \mathrm{diag}(\vm B_1, \vm B_2,..., \vm B_N )$ with $\vm B_i = \partialu\,\vm f_i(\vm 0,\vm 0,\vm 0)$. By Assumption \ref{ass:CLF} there exists a stabilizing controller $\vm u = \vm K \vm x$ such that the controlled linear system $\vm A + \vm B \vm K$ is stable. It is well known from MPC literature that if the inequality 
\begin{equation}
	(\vm A + \vm B \vm K)\trans \vm P + \vm P (\vm A+ \vm B\vm K) + \gamma (\vm Q + \vm K\trans \vm R \vm K) \leq 0 \label{eq:linear_CLF}
\end{equation} 
with $\gamma \in (1,\infty)$  is satisfied, then there exists a non-empty set $\Omega_{\beta} = \{\vm x \in \Gamma: \vm x\trans \vm P \vm x \leq \beta \}$ for some $\beta>0$ such that $\Omega_{\beta}$ is control invariant with the control law $\vm{\munderbar{u}} = \vm K \vm{\munderbar{x}}$, i.e. any initial state $\vm {\munderbar x}_0 \in \Omega_{\beta}$ implies that $ \vm {\munderbar x}(t) \in \Omega_{\beta}$ and $\vm{\munderbar{u}} (t) \in \mathbb{U}$ for all $t$ and that for any $ \vm {\munderbar x}(t) \in \Omega_{\beta}$, the CLF inequality $\frac{\dd }{\dd t} (\vm x\trans \vm P \vm x) +  \vm x\trans(\vm {\vm Q + \vm K\trans \vm R \vm K})\vm x \leq 0$ holds \cite{Chen,Mayne}. The requirements of Assumption \ref{ass:dist_CLF} and condition \eqref{eq:linear_CLF} can be reformulated as the following convex semi-definite optimization problem in $\vm P$ and $\vm K$
\begin{subequations}\label{eq:min_distCost}
	\begin{align}
		\min_{\vm E \succeq \vm 0,\, \vm Y}\quad& -\log (\det (\vm E)) \label{eq:distCost_costfunc} \\
		\stt \quad& \begin{bmatrix}
			\vm A \vm E + \vm E \vm A\trans + \vm B \vm Y + \vm Y\trans \vm B\trans & \vm E \vm Q^{\frac{1}{2}} & \vm Y\trans \vm R^{\frac{1}{2}} \\
			\vm Q^{\frac{1}{2}} \vm E & -\frac{1}{\gamma} \vm I & \vm 0 \\
			\vm R^{\frac{1}{2}} \vm Y & \vm 0 & - \frac{1}{\gamma}  \vm I
		\end{bmatrix}  \preceq \vm 0    \label{eq:distCost_CLF}\\
		\quad& \vm E = \mathrm{diag}\{ \vm E_i \, : \, \agents\} \label{eq:distCost_structure}\\
		\quad& \vm Y_{ij} = \vm 0\,, \quad \quad i \in \mathcal{V},\, j \notin \mathcal{N}_i\,, \label{eq:distcontroller_structure}
	\end{align}
\end{subequations}
with the substitutions $ \vm E:= \vm P^{-1}$ and $\vm Y:= \vm K \vm E$. The equivalence of conditions \eqref{eq:linear_CLF} and \eqref{eq:distCost_CLF} can be shown via Schur compliment techniques\cite{Boyd3}. For the considered quadratic cost functions, the constraint \eqref{eq:distCost_structure} ensures that the terminal cost $ V(\vm x)$ has the required separable structure, while the constraint \eqref{eq:distcontroller_structure} guarantees that the controller fulfills Assumption \ref{ass:dist_CLF}, i.e. $\vm r_i(\vm x_i, \vm x_{\mathcal{N}_i}) = \vm K_{ii} \vm x_i + \sum_{\neighs} \vm K_{ij} \vm x_j$. The cost objective requires maximizing $\log(\det \vm P^{-1})$ which is favorable in regard of enlarging the size of the terminal set as it maximizes the volume of the 1-level set ellipsoid $\Omega_{1} = \{\vm x \in \Gamma: \vm x\trans \vm P \vm x \leq 1 \}$ \cite{Conte}. Since problem \eqref{eq:min_distCost} is convex, global minimizerers $\vm P$ and $\vm K$ with the required structure exists as long as \eqref{eq:distCost_CLF} as well as $\vm E\succeq \vm 0$ are feasible. In this way, the optimization procedure \eqref{eq:min_distCost} provides a constructive test for the existence of a separable terminal cost acting as a CLF and a structured terminal controller as both required by Assumption~\ref{ass:dist_CLF}. Conservatism, however, is introduced by the choice of a block diagonal terminal cost and the structured linear controller when compared to the classical MPC approach without structural constraints. The size of the terminal region, i.e. $\beta$, can for example be determined similar to the procedure in \cite{Chen}.
%

\section{Distributed solution via sensitivities}
\label{sec:dist_sol}
The main idea of iterative and cooperative DMPC schemes is to solve~\eqref{eq:central_ocp} in a distributed fashion in each MPC step $k$. To this end, a sensitivity-based approach is used in which the local cost function \eqref{eq:agent_cost} of each agent is extended by so-called sensitivities of its neighbors such that the agents cooperatively solve the central OCP \eqref{eq:central_ocp}.

\subsection{Extended local optimal control problem}
The idea of utilizing first-order sensitivities to account for the influence of neighbors has been utilized before in the DMPC setting \cite{Huber,Huber2,Scheu,Scheu2,Pierer} and is derived for the problem at hand.
By extending the individual cost functions \eqref{eq:def_central_cost} of each agent $\agents$ with the first-order sensitivities of its neighbors $\neighs$, the agents not only consider their local cost objective but also a first-order approximation of their neighbors' costs. The sensitivity represents the information about the expected change in the overall cost objective \eqref{eq:central_costFunction} that the states of an agent $\agents$ induce via the local cost of agent $\rec$. The local problem of agent $\agents$ to be solved at step $q$ of the sensitivity-based algorithm therefore reads
\begin{subequations}\label{eq:local_ocp}
	\begin{align}
		\min_{ \vm u_i} \quad&
		\bar J_i(\vm u_i;\initstate):= V_i(\vm x_i(T)) + \int_0^T l_i(\vm x_i, \vm u_i, \Ni{x}\indqp) \,\dd \tau + \sum_{\rec} \delta \bar J_j(\vm u_j\indqp;\initstate )(\delta \vm x_i)
		\label{eq:local_costFunction}	\\
		~\stt \quad&	\vm{\dot{x}}_i = \vm f_i(\vm x_i, \vm u_i, \Ni{x}\indqp)\,, \quad\vm x(0) =\vm {\munderbar x}_{i,k}	\label{eq:local_ocp_dyn}\\
		&\vm u_i(\tau) \in \mathbb{U}_i\,, \quad \timetau
	\end{align}
\end{subequations}
with the neighbor's cost sensitivity represented by the corresponding Gâteaux derivative $  \delta \bar J_j(\vm u_j\indqp;\initstate )(\delta \vm x_i)$, $\rec$ in direction of $ \delta \vm x_i := \vm x_i - \vm x_i\indqp$ evaluated for the trajectories of the previous iteration $q-1$. The dynamics \eqref{eq:agent_dyn} and costs \eqref{eq:agent_cost} are decoupled by using the state trajectories $ \Ni{x}\indqp$ of the previous iteration which need to be communicated by each neighbor $\send$ to the agent $\agents$. The Gâteaux derivative $\delta \bar{J}_j $ in \eqref{eq:local_costFunction} can be expressed as (see Appendix A for an explicit derivation)
\begin{align} \label{eq:recursive_Gateaux}
	&\delta \bar J_j(\vm u_j\indqp;\initstate )(\delta \vm x_i)= \int_0^T \big(\underbrace{\partialx l_{ji}(\vm x_j\indqp, \vm x_i\indqp) + \partialx\,\vm f_{ji}(\vm x_j\indqp, \vm x_i\indqp)\trans \vm \lambda_j\indqp}_{ \vm g_{ji}\indqp(\tau)}\big)\trans\,\delta \vm x_i \, \dd \tau\,, \quad \rec
\end{align}
where $\vm \lambda_j \in \mathds{R}^{n_j}$ denotes the adjoint state of the neighbors $\rec$. Defining the local Hamiltonian for the local OCP \eqref{eq:local_ocp}
\begin{equation}\label{eq:local_Hamiltonian}
	H_i(\vm x_i, \vm u_i, \vm \lambda_i):= l_i(\vm x_i, \vm u_i,
	\Ni{x}\indqp )+\vm \lambda_i\trans \vm f_i(\vm x_i, \vm u_i,\Ni{x}\indqp)+ \sum_{\rec}  (\,\vm g_{ji}\indqp)\trans (\vm x_i - \vm x_i\indqp) \,,
\end{equation}
the adjoint state can be calculated via the backward integration of
\begin{equation} \label{eq:adjoint_dyn}
	\vm {\dot \lambda}_i(\tau) = - \partialx H_i(\vm x_i, \vm u_i, \vm \lambda_i)\,, 
	\quad  \vm \lambda_i(T) = \partialx V_i(\vm x_i(T; \initstate))\,, \quad \timetau \,.
\end{equation}
and needs to be communicated by each agent $\agents$ to its neighbors $\send$ such that  \eqref{eq:recursive_Gateaux} can be evaluated. 
Note that for sending neighbors $\send$ the sensitivity \eqref{eq:recursive_Gateaux} vanishes since $l_{ji} = 0$ and $\vm f_{ji} = \vm 0$. The bracketed term in \eqref{eq:recursive_Gateaux} can be interpreted as the (time-dependent) gradient $ \vm g_{ji}(\tau) \in \mathds{R}^{n_i}, \timetau$ of a neighbor's $\rec$ cost functional w.r.t. the (time-dependent) external trajectories $\vm x_i(\tau)$  of agent $\agents$. The gradient $\vm g_{ji}\indqp$ can be computed locally by each agent for its neighbors $\rec$ due to the neighbor-affine structure of \eqref{eq:agent_dyn} and \eqref{eq:agent_cost} as long as the agent has access to the coupling functions  $l_{ji}(\vm x_j, \vm x_i)$ and $\vm f_{ji}(\vm x_j, \vm x_i)$, $\rec$. 
\subsection{Sensitivity-based algorithm}
The separable structure of OCP~\eqref{eq:local_ocp} is exploited by solving the individual problems in parallel on the agent level. The procedure is summarized in Algorithm \ref{alg:SENSI_org} which shows the distributed sensitivity-based solution of the original central OCP \eqref{eq:central_ocp}. The algorithm is executed in each MPC step with the current system state $ \initstate = [ \vm{\munderbar x}_{i,k}]_{\agents} = \vm{\munderbar x}(t_k)$ in parallel and locally by each agent $\agents$. Algorithm \ref{alg:SENSI_org} consists of two local computations of which the first calculates the gradient $\vm g_{ji}$ for all neighbors $\rec$ and the second requires the solution of the local OCP~\eqref{eq:local_ocp}. In addition, one communication step is needed in which the state and adjoint trajectories are exchanged between neighboring agents. Thus, Algorithm \ref{alg:SENSI_org} constitutes a fully distributed algorithm with only local computation steps and neighbor-to-neighbor communication. For an efficient solution of OCP \eqref{eq:alg_OCP} the fixed-point iteration scheme\cite{Pierer3} or the projected gradient method\cite{Graichen3,Englert} can be used which also conveniently calculate the adjoint state $ \vm \lambda_i$. Alternatively, a standard OCP solver can be used to compute $(\vm u_i^k,\vm x_i^k)$ in \eqref{eq:alg_OCP}, but then the adjoint dynamics~\eqref{eq:adjoint_dyn} must be solved backward in time to obtain $\vm\lambda_i^k$.
The stopping criterion in Step 4 terminates the algorithm either if a maximum number of iterations $q_{\mathrm{max}}$ has been reached or if a convergence criterion is satisfied at a certain iteration number $q_k$ in the current MPC step. The criterion \eqref{eq:stopp_crit} evaluates the progress of the state and adjoint state trajectories ($\vm x_i^q(\tau), \vm \lambda_i^q(\tau)$) between two algorithm iterations w.r.t. the norm of the current system state $\initstate$ weighted by the constant $d>0$. This leads to a contraction of the stopping criterion during the stabilization of the system to the origin. Although the stopping criterion is only evaluated locally, it needs to be satisfied by all agents simultaneously in each MPC step which requires either a central node or a global communication step. In a practical implementation, this convergence monitoring is usually done by a central node which is also responsible for triggering the individual algorithm steps in a synchronized manner.\cite{Burk} 
This central convergence check can be avoided if instead a maximum number of iterations $q_{\mathrm{max}}$ is performed without evaluating the stopping criterion. In this way, the convergence criterion \eqref{eq:stopp_crit} can be used if a certain accuracy is required, while stopping after a maximum number of iterations is beneficial for real-time applications with a fixed-sampling time $ \Delta t$. In the next section, conditions on the stopping criterion constant $d$ and the maximum number of iterations $q_{max}$ are derived such that the resulting DMPC controller stabilizes the system exponentially. 
In the first MPC step $k=0$, Algorithm \ref{alg:SENSI_org} is locally initialized with appropriate trajectories, e.g. $ \vm x_i^0( \tau; \vm x_0) = \vm{\munderbar x}_{i,0}$ and $ \vm \lambda_i^0( \tau, \vm x_0) = \partialx V_i(\vm{\munderbar x}_{i,0})$. In the following MPC steps $k \geq 1$, the algorithm is warm-started by the solution trajectories of the last iteration $ q_k$ in the previous MPC step, i.e.
\begin{align} \label{eq:warm_start}
	\vm x_i^0 ( \tau;  \vm{\munderbar x}_{k+1}) = \vm x_i\indqk ( \tau;  \vm{\munderbar x}_{k})\,, \quad \vm \lambda_i^0 ( \tau;  \vm{\munderbar x}_{k+1}) = \vm \lambda_i\indqk ( \tau; \vm{\munderbar x}_{k})\,, \quad \timetau \,, \quad \agents
\end{align}
for all agents $\agents$. As soon as the stopping criterion \eqref{eq:stopp_crit} is fulfilled for all agents $\agents$ or the maximum number of iterations $q_{\mathrm{max}}$ are reached, the calculated control trajectory $ \vm u_i\indqk(\tau; \initstate)$ from the last iteration $q_k$ is taken as the input for the actual subsystems \eqref{eq:agent_dyn}, i.e.
\begin{equation}
	\vm{\munderbar u}_i(t_k + \tau ) = \vm u_i\indqk(\tau; \initstate)\,, \quad \timetausample\,, \quad \agents.
\end{equation}
Similar to \eqref{eq:MPC_control_law}, the current controls $ \vm u\indqk(\tau; \initstate) = [ \vm u_i]_{\agents}$ can interpreted as the nonlinear sampled control law
\begin{equation} \label{eq:DMPC_control_law}
	\vm u\indqk(\tau; \initstate) =: \vm \kappa( \vm x\indqk(\tau; \initstate); \vm x\indqkp(\tau; \initstate), \vm \lambda\indqkp(\tau; \initstate), \initstate )
\end{equation}
with $ \timetausample$. Note, however, in contrast to the centralized MPC case \eqref{eq:MPC_control_law}, the control law \eqref{eq:DMPC_control_law} is parameterized by the state and adjoint trajectories $ \vm x\indqkp(\tau; \initstate)$ and  $ \vm \lambda\indqkp(\tau; \initstate)$ of the previous iteration of Algorithm \ref{alg:SENSI_org}.
\begin{algorithm}[tb]
	\caption{Distributed sensitivity-based solution of OCP \eqref{eq:central_ocp}}
	\begin{algorithmic}[1]
		\Statex Initialize $(\vm x_i^{0}$, $\vm \lambda_i^{0})$, set $q=1$, choose $d>0$ or $q_{\mathrm{max}}$, send $\vm x_i^{0}$ to $j \in \mathcal{N}_i$ and $ \vm \lambda_i^{0}$ to $j \in \mathcal{N}_i^\leftarrow$ 
		\State Compute the gradients $\vm g_{ji}\indqp$ for all $\rec$ as
		\begin{align} 
			\vm g_{ji}\indqp(\tau)=\partialx l_{ji}(\vm x_j\indqp, \vm x_i\indqp) + \partialx\,\vm f_{ji}(\vm x_j\indqp, \vm x_i\indqp)\trans \vm \lambda_j\indqp\,, \quad \timetau
		\end{align}
		\State Compute $(\vm u_i^{q}, \vm x_i^{q}, \vm\lambda_i^{q})$ by solving 
		\begin{subequations}\label{eq:alg_OCP}
			\begin{align}
				\min_{ \vm u_{i}} \quad & V_i(\vm x_i(T)) + \int_0^T l_i(\vm x_i, \vm u_i, \Ni{x}\indqp)+\!\! \sum_{\rec} (\vm g_{ji}\indqp)\trans(\vm x_i - \vm x_i\indqp) \,\dd\tau 
				\\  ~\stt \quad&\vm{\dot x}_i = \vm f_i(\vm x_i, \vm u_i,\Ni{x}\indqp)\,, 
				\quad \vm x_i(0) =\vm{\munderbar x}_{i,k}\\
				&\vm u_i(\tau) \in \mathbb{U}_i\,, \quad \timetau
			\end{align}
		\end{subequations}
		\State Send state trajectories $\vm x_i^{q}$ to $\neighs$ and adjoint trajectories $\vm \lambda_i\indq$ to $\send$ 
		\State \textbf{if} $q = q_{\mathrm{max}}$ or \begin{equation}\label{eq:stopp_crit}
			\bigg\|\begin{matrix}
				\vm x_i^q - \vm x_i^{q-1} \\
				\vm \lambda_i^q - \vm \lambda_i^{q-1} \\
			\end{matrix}\bigg\|_{L_\infty}\leq d \| \munderbar{\vm x}_{i,k} \|\,,\quad \forall \agents \end{equation}
		\State\textbf{then} Quit with $q_k := q$
		\State \textbf{else} Increment $q \leftarrow q + 1$ and go to Step 1
	\end{algorithmic}\label{alg:SENSI_org}
\end{algorithm}
%
%

\section{DMPC stability analysis under inexact optimization}
\label{sec:DMPC_stab}
By construction, Algorithm \ref{alg:SENSI_org} iterates until a suboptimal input $\vm u_i\indqk(\tau; \initstate)$, $\timetausample$ is found, either when the stopping criterion \eqref{eq:stopp_crit} is fulfilled or a maximum number of iterations $q_{\mathrm{max}}$ is reached. While Algorithm \ref{alg:SENSI_org} is consequently suitable for a real-time capable implementation, stability of the DMPC scheme is not guaranteed since Corollary \ref{th:dist_opt_MPC} does not apply anymore. That is why stability results for the DMPC scheme under an inexact minimization of the central OCP \eqref{eq:central_ocp} by Algorithm \ref{alg:SENSI_org} are analyzed in this section. Two scenarios are considered: First an upper bound on the constant $d$ in \eqref{eq:stopp_crit} is derived in Section \ref{subsec:stab_analys} and then the bound is reformulated in Section \ref{subsec:fixed_iter} in terms of the required number of iterations $q_{\mathrm{max}}$ that must be executed in each MPC step. The proof follows along the lines of the work\cite{Bestler}, where the (D)MPC stability of a similar scheme, based on ADMM, was investigated. However, we require less assumptions and in particular do not need to assume convergence and boundedness of the iterates of the employed distributed optimization algorithm but rather are able to show linear convergence of Algorithm \ref{alg:SENSI_org} which turns out to be critical for the subsequent stability analysis of the DMPC scheme. In contrast the trade-off is that convergence of Algorithm \ref{alg:SENSI_org} can only be guaranteed for some maximum horizon length which is investigated in more detail in Lemma \ref{lem:lin_conv} at the end of Section \ref{subsec:prel}. 
\subsection{Preliminaries for the suboptimal case}
\label{subsec:prel}
Compared to the optimal (D)MPC feedback laws \eqref{eq:MPC_control_law} and \eqref{eq:dist_opt_controllaw}, the control \eqref{eq:DMPC_control_law} represents a suboptimal feedback law. This is caused by prematurely stopping Algorithm \ref{alg:SENSI_org} after a certain number of iterations. As a consequence, the solution trajectories of the sensitivity-based DMPC algorithm are not consistent with the optimal MPC solution and could lead to a possibly destabilizing solution as Theorem \ref{th:centralMPC} or Corollary \ref{th:dist_opt_MPC} do not hold anymore. 

While stopping Algorithm \ref{alg:SENSI_org} prematurely limits the number of required iterations and thus the time-expensive communication steps, it also leads to suboptimal control trajectories $ \vm u\indqk(\tau; \vm {\munderbar x}_k)$ that differ from the optimal ones $ \vm u\inds(\tau; \vm {\munderbar x}_k)$. That is why in general, the state trajectories of the local MPC predictions as part of the solution of the local OCP \eqref{eq:alg_OCP}, the actual realizations from applying the suboptimal control \eqref{eq:DMPC_control_law} of each agent to the actual system \eqref{eq:centr_dyn}, and the optimal ones resulting from the centralized MPC solution \eqref{eq:central_ocp} will not be identical. In particular, we have to distinguish between the following trajectories:
\begin{itemize}
	\item The individual predicted state trajectories $\vm x\indqk(\cdot; \initstate) = [\vm x_i\indqk (\cdot; \initstate)]_{\agents}$ resulting from the solution of OCP \eqref{eq:alg_OCP} in the last iteration $q_k$, i.e.
	\begin{equation}\label{eq:pred_traj}
		\vm{ \dot x}\indqk(\tau; \initstate) = \vm{ \hat f}(\vm x\indqk(\tau; \initstate), \vm u\indqk(\tau; \initstate),\vm{\hat x}\indqkp(\tau; \initstate))\,, \quad \vm x\indqk(0; \initstate)  = \initstate
	\end{equation}
	with the stacked notation $\vm {\hat x}\indqkp(\tau; \initstate) = [\Ni{x}\indqkp]_{\agents} \in \mathds{R}^{p}$ and $\vm{ \hat f} = [ \vm f_i ]_{\agents}$, where the notation $\vm{ \hat f}$ explicitly captures the dependency on $ \vm {\hat x}\indqkp(\tau; \initstate)$.
	\item The actual state trajectories $\vm x_c\indqk(\cdot; \initstate)$ resulting from applying the suboptimal controls $\vm u\indqk(\cdot; \initstate) = [\vm u_i\indqk(\cdot; \initstate)]_{\agents}$ to the central system \eqref{eq:centr_dyn}, i.e.
	\begin{equation}\label{eq:act_traj}
		\vm {\dot x}_c\indqk(\tau; \initstate) = \vm f(\vm x_c\indqk(\tau; \initstate), \vm u\indqk(\tau; \initstate))\,, \quad  \vm x_c\indqk(0; \initstate) = \initstate\,.
	\end{equation}
	\item The optimal trajectories $\vm x\inds(\cdot; \initstate) = [\vm x_i\inds (\cdot; \initstate)]_{\agents}$ following from solving the central OCP \eqref{eq:central_ocp} cf. \eqref{eq:opt_statedyn}, i.e.
	\begin{equation}\label{eq:opt_traj}
		\vm{\dot x}\inds(\tau; \initstate) = \vm f(\vm x\inds(\tau; \initstate), \vm u\inds(\tau; \initstate))\,, \quad  \vm x\inds(0; \initstate) = \initstate\,.
	\end{equation}
\end{itemize}
In the nominal case, this implies that the system state of~\eqref{eq:centr_dyn} in the next MPC step, i.e. $\vm{\munderbar x}_{k+1} =  \vm{\munderbar x}(t_{k+1})$, will lie on the actual state trajectory 
\begin{equation} \label{eq:next_sampling_step}
	\vm{\munderbar x}_{k+1} = \vm x_c\indqk(\Delta t;\initstate)
\end{equation} 
and not on the individual state trajectory $\vm x\indqk(\Delta t; \initstate)$ or the optimal one $\vm x\inds(\Delta t; \initstate)$.
This difference between optimal and actual state trajectory is expressed by the error between actual state trajectory and optimal state trajectory 
\begin{equation} \label{eq:error}
	\Delta \vm x_c\indqk(\tau; \initstate) : = \vm x_c\indqk(\tau; \initstate) - \vm x\inds(\tau; \initstate)\,, \quad \timetau
\end{equation}
that can be interpreted as a suboptimality measure in each MPC step $k$. Several assumptions and intermediate results concerning the optimal solution of the central MPC scheme and the properties of Algorithm \ref{alg:SENSI_org} are necessary to proceed.
\begin{assumption}\label{ass:cost_Lipschitz}
	The optimal cost \eqref{eq:opt_cost} is twice continuously differentiable and the feedback laws \eqref{eq:MPC_control_law} and \eqref{eq:DMPC_control_law} are locally Lipschitz w.r.t. their arguments. 
\end{assumption}
These continuity assumptions are typical in (D)MPC \cite{Graichen,Graichen2} and are required to derive certain bounds in the stability analysis. A consequence of Assumption \ref{ass:cost_Lipschitz} is that there exist constants $m_J$, $M_J>0$ such that the optimal cost satisfies (cf. \eqref{eq:upper_bound_opt_cost} and \eqref{eq:lower_bound_opt_cost} in Appendix E)
\begin{equation}\label{eq:cost_bound}
	m_J \| \initstate\|^2 \leq J\inds(\initstate) \leq M_J \| \initstate\|^2\,, \quad \forall \initstate \in \Gamma_\alpha\,.
\end{equation}
In order to proceed with the stability analysis, the rate of convergence of Algorithm \ref{alg:SENSI_org} needs to be characterized. The next lemma shows linear convergence of the sensitivity-based DMPC Algorithm. 
\begin{lemma}\label{lem:lin_conv}
	Suppose that Assumption  \ref{ass:cost_Lipschitz} holds. Then, there exists an upper bound on the horizon length $T_{\max}>0$ such that for $T<T_{\max}$, Algorithm \ref{alg:SENSI_org} converges linearly, i.e.
	\begin{align}\label{eq:lin_conv}
		\bigg\|\begin{matrix}
			\vm x\inds(\cdot; \initstate) - \vm x^{q}(\cdot; \initstate) \\
			\vm \lambda\inds(\cdot; \initstate) - \vm \lambda^{q}(\cdot; \initstate) \\
		\end{matrix}\bigg\|_{L_\infty} \leq p \bigg\|\begin{matrix}
			\vm x\inds(\cdot; \initstate) - \vm x^{q-1}(\cdot; \initstate) \\
			\vm \lambda\inds(\cdot; \initstate) - \vm \lambda^{q-1}(\cdot; \initstate) \\
		\end{matrix}\bigg\|_{L_\infty}
	\end{align}
	for some $p \in(0,1)$.
\end{lemma}
\begin{proof}
	See Appendix B. 
\end{proof} 
Lemma \ref{lem:lin_conv} reveals the fact that convergence can be guaranteed with a sufficiently small prediction horizon $T$. This shows that there exists a trade-off between the size of the domain of attraction $\Gamma_{\alpha}$ which can be enlarged by increasing $T$, cf. Theorem \ref{th:centralMPC}, and the stability of the distributed algorithm which can be achieved by decreasing $T$. It should be pointed out that no convexity assumptions on the original OCP \eqref{eq:central_ocp} are necessary and convergence is explicitly shown without critical assumptions which sets it apart from, e.g. ADMM.\cite{Bestler}
\begin{remark}
In order to improve convergence and enlarge the maximum horizon $T_{\mathrm{max}}$ the iterates of Algorithm \ref{alg:SENSI_org} can be damped as first proposed in \cite{Graichen,Pierer3}, i.e. by performing 
\begin{align}\label{eq:damping}
	\vm x_i\indq(\tau)\leftarrow(1-\epsilon)\vm x_i\indq(\tau) + \epsilon\vm x_i\indqp(\tau)\,, \quad 
	\vm \lambda_i\indq(\tau)\leftarrow(1-\epsilon)\vm \lambda_i\indq(\tau) + \epsilon\vm \lambda_i\indqp(\tau)\,, \quad \timetau
\end{align}
as an intermediate step between Step 2 and 3 of Algorithm \ref{alg:SENSI_org} with damping factor $\epsilon \in [0,1)$. This measure seeks to prevent drastic changes during the iterations and is therefore beneficial as the sensitivities can be interpreted as a first-order approximation of the change in costs and has a strong impact on the admissible horizon \cite{Pierer3}. 
\end{remark}
\subsection{Convergence properties with stopping criterion}
\label{subsec:stab_analys}
The investigation of stability of the DMPC scheme requires to look at the error $\Delta \vm x_c\indqk(\tau; \initstate)$ in \eqref{eq:error} as the difference between the suboptimal sensitivity-based and the optimal (D)MPC solution. The following lemma bounds this error in terms of the stopping criterion \eqref{eq:stopp_crit} for all agents $i \in \mathcal{V}$.
\begin{lemma}\label{lem:error_bound}
	Suppose that Assumption \ref{ass:cost_Lipschitz} and $T<T_{\mathrm{max}}$ hold. Then, there exists a constant $D>0$ such that the error \eqref{eq:error} satisfies
	\begin{equation}\label{eq:error_bound}
		\| \Delta \vm x_c\indqk(\cdot;\initstate)\|_{L_\infty}\leq D d\|\initstate \|\,, \quad \forall \initstate \in \Gamma_\alpha\,.
	\end{equation}
\end{lemma}
\begin{proof}
	See Appendix C.
\end{proof}
In the centralized optimal MPC case, the next sampling point lies on the optimal state trajectory, i.e. $\vm{\munderbar x}_{k+1} = \vm x^*(\Delta t; \initstate)$ and the  cost $J\inds(\vm x^*(\Delta t; \initstate))$ decreases according to \eqref{eq:th1}. As indicated by \eqref{eq:next_sampling_step}, this is not the case for the suboptimal distributed MPC scheme. The next lemma relates the cost decrease in the suboptimal case to the optimal cost $J\inds(\initstate)$ and the error \eqref{eq:error}.
\begin{lemma}\label{lem:cost_bound}
	Suppose that Assumptions \ref{ass:CLF} to \ref{ass:cost_Lipschitz} hold. Then, there exists some $\bar{ \alpha}<\alpha$ and correspondingly a subset $\Gamma_{\bar \alpha} \subset \Gamma_\alpha$ such that $\vm x_c\indqk(\Delta t;\initstate) \in \Gamma_\alpha$ for all $\initstate \in \Gamma_{\bar \alpha}$. Moreover, there exist constants $0< a \leq 1$ and $b,c>0$ such that the optimal cost at the next sampling point $t_{k+1}$  satisfies 
	\begin{equation} \label{eq:lemma_2}
		J\inds(	\vm x_c\indqk(\Delta t;\initstate)) \leq (1-a)J^*(\initstate)  + b \sqrt{J^*(\initstate)} \|\Delta \vm x_c\indqk(\cdot;\initstate)\|_{L_\infty} + c \|\Delta \vm x_c\indqk(\cdot;\initstate)\|_{L_\infty}^2\,,\quad \forall\initstate \in \Gamma_{\bar \alpha}\,.
	\end{equation}
\end{lemma}
\begin{proof}
	See Appendix D.
\end{proof}
Lemma \ref{lem:cost_bound} reveals that the cost decrease and overall MPC performance is limited by the error $\|\Delta \vm x_c\indqk(\cdot;\initstate)\|_{L_\infty}$ which opposes the contraction term $(1-a)J^*(\initstate)$. Moreover, the original domain of attraction $\Gamma_{\alpha}$ is reduced to the smaller set $\Gamma_{\bar \alpha}$. Based on Lemmas \ref{lem:error_bound} and \ref{lem:cost_bound}, the stability of the DMPC scheme under inexact minimization can be shown.
\begin{theorem} \label{th:sub_MPC}
	Suppose that Assumptions \ref{ass:CLF} through \ref{ass:cost_Lipschitz} as well as $T< T_{\mathrm{max}}$ hold and let the constant $d$ in the stopping criterion~\eqref{eq:stopp_crit} satisfy 
	\begin{equation} \label{eq:th2}
		d < \frac{\sqrt{m_J}}{2cD}\bigg( \sqrt{b^2 + 4ac} - b\bigg)\,.
	\end{equation}
	Then, the optimal cost \eqref{eq:opt_cost} and the error \eqref{eq:error} decay exponentially for all $\vm {\munderbar x}_0 \in \Gamma_{\bar \alpha}$ and the origin of the closed-loop system under the DMPC control law \eqref{eq:DMPC_control_law} is exponentially stable.
\end{theorem}
\begin{proof}
	The result of Lemma \ref{lem:error_bound} can be expressed in terms of the optimal cost $J\inds(\initstate)$ using the quadratic bound \eqref{eq:cost_bound}, i.e.
	\begin{equation} \label{eq:bound_delta_x_opt_cost}
		\| \Delta \vm  x_c\indqk(\cdot;\initstate)\|_{L_\infty} \leq \frac{D d}{\sqrt{m_J}} \sqrt{J\inds(\initstate)}
	\end{equation}
	which, inserted in the relation \eqref{eq:lemma_2} of Lemma \ref{lem:cost_bound}, results in a bound on the cost in the next MPC step 
	\begin{align}\label{eq:contr_optimalCost}
		J\inds(	\vm x_c\indqk(\Delta t;\initstate) \leq p_J J\inds(\initstate)
	\end{align} 
	with $p_J:= (1-a) + \frac{bDd}{\sqrt{m_J}} + \frac{cD^2d^2}{m_J}$. By \eqref{eq:th2}, the contraction ratio satisfies $p_J<1$ which in return implies the exponential decay of the optimal cost $J\inds(\vm {\munderbar x}_k)= J\inds(\vm x_c\indqk(\Delta t;\vm x_{k-1}))$
	\begin{align} \label{eq:cost_decrease}
		J\inds(\initstate) \leq  p_J J\inds(\vm {\munderbar x}_{k-1})  \leq (p_J)^{k} J^*(\vm {\munderbar x}_0) \leq (p_J)^{k}\bar{\alpha}
	\end{align}
	and of the error $\| \Delta \vm  x_c\indqk(\cdot;\initstate)\|_{L_\infty}$ by \eqref{eq:bound_delta_x_opt_cost}
	\begin{equation} \label{eq:bound_on_error}
		\| \Delta \vm  x_c\indqk(\cdot;\initstate)\|_{L_\infty} \leq \frac{D d}{\sqrt{m_J}}  \sqrt{ (p_J)^k \bar{ \alpha}}\,.
	\end{equation}
	To show exponential stability in continuous time, the state trajectory of the closed loop, i.e. $\vm{\munderbar x}(t) = \vm{\munderbar x}(t_k + \tau) = \vm x_c\indqk(\tau; \initstate)$, $ \timetausample$, $k \in \mathds{N}_0$ needs to be bounded by an exponential envelope function. Using the triangle inequality, the bound \eqref{eq:state_upper_bound}, and the bound \eqref{eq:bound_on_error}, results in
	\begin{align} \label{eq:cont_time_stability}
		\|  \vm x_c\indqk(\tau; \initstate)\| &\leq \| \vm x\inds(\tau; \initstate)\| + \|\Delta \vm x\indqk_c(\tau; \initstate)\| \leq \|\initstate\| \mathrm{e}^{\hat L \tau } + \|\Delta \vm x\indqk_c(\cdot; \initstate)\|_{L_\infty} \nonumber\leq  (\mathrm{e}^{\hat L \tau } + D d ) \|\initstate\| \\
		&\leq  (\mathrm{e}^{\hat L \tau } + D d ) \frac{1}{\sqrt{m_J}} \sqrt{J\inds(\initstate)}\leq  (\mathrm{e}^{\hat L \tau } + D d ) \frac{1}{\sqrt{m_J}} (p_J)^{\frac{1}{2}k} \sqrt{J\inds(\vm {\munderbar x}_0)}
	\end{align}
	which bounds the state trajectory in each MPC step $k$ for $\timetausample$. The exponential decay of \eqref{eq:cont_time_stability} implies that there exists constants $\gamma_1,\, \gamma_2>0$ such that $\|\vm x(t)\| \leq \gamma_1 \mathrm{e}^{- \gamma_2 t}$.
\end{proof}
Theorem \ref{th:sub_MPC} shows that an explicit value for $d$ can be computed offline ensuring stability and incremental improvement of the DMPC scheme. In particular, it is evident that the value of $d>0$ controls the contraction rate $p_J$ in \eqref{eq:contr_optimalCost} and the optimal (D)MPC case, i.e. $p_J \rightarrow (1-a)$, is recovered for $d \rightarrow 0$. In fact, it is possible to explicitly control the worst-case contraction rate $p_J$ in the interval $(1-a,1)$ via $d$.  However, calculating the value of $d$ via \eqref{eq:th2} is usually too conservative for design purposes due to the various involved Lipschitz and continuity estimates. Nevertheless, Theorem \ref{th:sub_MPC} states that a stabilizing $d$ can always be found.  

\begin{remark}
Compared to the scheme given in the work\cite{Bestler}, we do not require to assume boundedness and linear convergence of the underlying distributed optimization algorithm to show stability of the MPC scheme, but explicitly prove that there exists an upper bound on the horizon such that the iterates are bounded and linear convergence of Algorithm \ref{alg:SENSI_org} is ensured. The drawback is that the horizon cannot be chosen arbitrarily long which can possibly decrease the performance of the (D)MPC-controller. However, this can be mitigated by damping the trajectories, cf. \eqref{eq:damping}.
\end{remark}

\subsection{Convergence properties with a fixed number of iterations}
\label{subsec:fixed_iter}
In the following, it is discussed how the stopping criterion \eqref{eq:stopp_crit} can be equivalently fulfilled with a finite number of iterations $q$ without explicitly evaluating it. The following theorem precisely states a lower bound on the required iterations $q_{\mathrm{max}}$. 
\begin{theorem} \label{th:sub_MPC_num_iter}
	Suppose that Assumptions \ref{ass:CLF} through \ref{ass:cost_Lipschitz} as well as $T< T_{\mathrm{max}}$ hold and $d$ is chosen according to Theorem \ref{th:sub_MPC}. Then, for all $\vm {\munderbar x}_0 \in \Gamma_{\bar \alpha}$ there exists a constant $e>0$ such that if the number of iterations per MPC step and the initial optimization error satisfy 
	\begin{equation}\label{eq:th3}
		q_{\mathrm{max}}> 1 + \log_p\bigg(\frac{d}{e(1+p)}\bigg)\,, \quad \bigg\|\begin{matrix}
			\vm x^1(\cdot; \vm {\munderbar x}_0) - \vm x^{0}(\cdot; \vm {\munderbar x}_0) \\
			\vm \lambda^1(\cdot; \vm {\munderbar x}_0) - \vm \lambda^{0}(\cdot; \vm {\munderbar x}_0) \\
		\end{matrix}\bigg\|_{L_\infty} \leq d N \sqrt{\frac{\bar \alpha}{m_J}} p^{1-q_{\mathrm{max}}}\,,
	\end{equation}
	then the origin of the closed-loop system under the DMPC control law \eqref{eq:DMPC_control_law} is exponentially stable and the optimization error~\eqref{eq:error} decays exponentially.
\end{theorem} 
\begin{proof}
	The proof consists of showing by induction that the stopping criterion, i.e. 
	\begin{equation} \label{eq:stopping_critertion_th3}
		\bigg\| \begin{matrix}
			\delta \vm x^q(\cdot;\initstate)\\ \delta \vm \lambda^q(\cdot;\initstate)
		\end{matrix} \bigg \|_{L_\infty} \leq d \| \initstate\|\,,
	\end{equation}
	with $\delta \vm x^q(\tau;\initstate) :=  \vm x\indq(\tau; \initstate) - \vm x\indqp(\tau; \initstate)$ and $\delta \vm \lambda^q(\tau;\initstate) :=  \vm \lambda\indq(\tau; \initstate) - \vm \lambda\indqp(\tau; \initstate)$ holds in each MPC step $k$ for the given conditions \eqref{eq:th3}. Consequently, the optimal cost $J\inds(\initstate)$ and optimization error $\| \Delta \vm  x_c\indqk(\cdot;\initstate)\|_{L_\infty}$ decrease according to \eqref{eq:cost_decrease} and \eqref{eq:bound_on_error} as stated by Theorem \ref{th:sub_MPC}. Note that the linear convergence property \eqref{eq:lin_conv} can equivalently be expressed as 
	\begin{equation} \label{eq:liner_convergence}
		\bigg\| \begin{matrix}
			\delta \vm x^q(\cdot;\initstate)\\ \delta \vm \lambda^q(\cdot;\initstate)
		\end{matrix} \bigg \|_{L_\infty} \leq p^{q-1} \bigg\| \begin{matrix}
			\delta \vm x^1(\cdot;\initstate)\\ \delta \vm \lambda^1(\cdot;\initstate)
		\end{matrix} \bigg \|_{L_\infty}.
	\end{equation}
	The relation \eqref{eq:stopping_critertion_th3} must hold at $k=0$ (induction start), i.e. 
	\begin{equation}
		\bigg\| \begin{matrix}
			\delta \vm x^q(\cdot;\vm {\munderbar x}_0)\\ \delta \vm \lambda^q(\cdot;\vm {\munderbar x}_0)
		\end{matrix} \bigg \|_{L_\infty} \leq p^{q-1}\bigg\| \begin{matrix}
			\delta \vm x^1(\cdot;\vm {\munderbar x}_0)\\ \delta \vm \lambda^1(\cdot;\vm {\munderbar x}_0)
		\end{matrix} \bigg \|_{L_\infty} \leq N d \| \vm {\munderbar x}_0\|
	\end{equation}
	which in return shows that the initial optimization error must satisfy 
	\begin{equation}
		\bigg\| \begin{matrix}
			\delta \vm x^1(\cdot;\vm {\munderbar x}_0)\\ \delta \vm \lambda^1(\cdot;\vm {\munderbar x}_0)
		\end{matrix} \bigg \|_{L_\infty} \leq d N \sqrt{\frac{\bar{\alpha}}{m_J}} p^{1-q_{\mathrm{max}}}
	\end{equation}
	which is true by assumption in Theorem \ref{th:sub_MPC_num_iter}. Following the lines of Theorem \ref{th:sub_MPC}, this implies that $\| \Delta \vm x_c\indqk(\cdot;\vm {\munderbar x}_0)\|_{L_\infty}\leq D d\|\vm {\munderbar x}_0\| \leq   d\frac{D}{\sqrt{m_J}} \sqrt{\bar{\alpha}} $ which in return shows that $J\inds(\vm {\munderbar x}_1)\leq p_J J\inds(\vm {\munderbar x}_0) $ completing the induction start. We continue with the induction hypothesis that \eqref{eq:stopping_critertion_th3} holds at MPC step $k$. At MPC step $k+1$ the following bound can be given via the linear convergence property \eqref{eq:liner_convergence} and Minkowski's inequality
	\begin{align} \label{eq:induction_step}
		&\bigg\| \begin{matrix}
			\delta \vm x^q(\cdot;\vm {\munderbar x}_{k+1})\\ \delta \vm \lambda^q(\cdot;\vm {\munderbar x}_{k+1})
		\end{matrix} \bigg \|_{L_\infty} \leq  \bigg\| \begin{matrix}
			\vm x^q(\cdot; \vm {\munderbar x}_{k+1})- \vm x\inds(\cdot; \vm {\munderbar x}_{k+1})\\ \vm \lambda^q(\cdot; \vm {\munderbar x}_{k+1}) -\vm \lambda\inds(\cdot; \vm {\munderbar x}_{k+1})
		\end{matrix} \bigg \|_{L_\infty} + \bigg\| \begin{matrix}
			\vm x^{q-1}(\cdot; \vm {\munderbar x}_{k+1})- \vm x\inds(\cdot; \vm {\munderbar x}_{k+1})\\ \vm \lambda^{q-1}(\cdot; \vm {\munderbar x}_{k+1}) -\vm \lambda\inds(\cdot; \vm {\munderbar x}_{k+1})
		\end{matrix} \bigg \|_{L_\infty} \nonumber \\
		&\leq  (1 + p)p^{q-1} \bigg\| \begin{matrix}
			\vm x^0(\cdot; \vm {\munderbar x}_{k+1})- \vm x\inds(\cdot; \vm {\munderbar x}_{k+1})\\ \vm \lambda^0(\cdot; \vm {\munderbar x}_{k+1}) -\vm \lambda\inds(\cdot; \vm {\munderbar x}_{k+1})
		\end{matrix} \bigg \|_{L_\infty}  
		\leq (1 + p)p^{q-1} \bigg(  \bigg\| \begin{matrix}
			\vm x^0(\cdot; \vm {\munderbar x}_{k+1})- \vm x\inds(\cdot; \vm {\munderbar x}_k)\\ \vm \lambda^0(\cdot; \vm {\munderbar x}_{k+1}) -\vm \lambda\inds(\cdot; \vm {\munderbar x}_k)
		\end{matrix} \bigg \|_{L_\infty} + \bigg\| \begin{matrix}
			\vm x\inds(\cdot; \vm {\munderbar x}_k)- \vm x\inds(\cdot; \vm {\munderbar x}_{k+1})\\ \vm \lambda\inds(\cdot; \vm {\munderbar x}_k) -\vm \lambda\inds(\cdot; \vm {\munderbar x}_{k+1})
		\end{matrix} \bigg \|_{L_\infty} \bigg).
	\end{align}
	We now bound the first norm on the right hand side of \eqref{eq:induction_step} by using the re-initialization for the adjoint state $ \vm \lambda^0(\tau; \vm {\munderbar x}_{k+1}) = \vm \lambda^q(\tau; \vm {\munderbar x}_k)$  and state $\vm x^0(\tau; \vm {\munderbar x}_{k+1}) = \vm x^q(\tau; \vm {\munderbar x}_k)$ as follows
	\begin{align} \label{eq:suboptimal_bound}
		\bigg\| \begin{matrix}
			\vm x^0(\cdot; \vm {\munderbar x}_{k+1})- \vm x\inds(\cdot; \vm {\munderbar x}_k)\\ \vm \lambda^0(\cdot; \vm {\munderbar x}_{k+1}) -\vm \lambda\inds(\cdot; \vm {\munderbar x}_k)
		\end{matrix} \bigg \|_{L_\infty} = \bigg\| \begin{matrix}
			\vm x^q(\cdot; \vm {\munderbar x}_k)- \vm x\inds(\cdot; \vm {\munderbar x}_k)\\ \vm \lambda^q(\cdot; \vm {\munderbar x}_k) -\vm \lambda\inds(\cdot; \vm {\munderbar x}_k)
		\end{matrix} \bigg \|_{L_\infty} \leq \frac{p}{1-p} \bigg\| \begin{matrix}
			\vm x^q(\cdot; \vm {\munderbar x}_k)- \vm x\indqp(\cdot; \vm {\munderbar x}_k)\\ \vm \lambda^q(\cdot; \vm {\munderbar x}_k) -\vm \lambda\indqp(\cdot; \vm {\munderbar x}_k)
		\end{matrix} \bigg \|_{L_\infty}  \leq \frac{p}{1-p}Nd \|\vm {\munderbar x}_k\|\,,
	\end{align}
	where the induction hypothesis \eqref{eq:stopping_critertion_th3} was used in the last inequality of \eqref{eq:suboptimal_bound}. Regarding the second norm on the right hand side of \eqref{eq:induction_step}, Equations  \eqref{eq:optstate_bound} and \eqref{eq:bound_opt_lambda} in Appendix E show that there exists Lipschitz constants $L_{x},L_{\lambda}>0$ such that  $\|\vm x\inds(\cdot; \vm {\munderbar x}_{k+1})- \vm x\inds(\cdot; \vm {\munderbar x}_{k})\|_{L_\infty} \leq L_{x} \| \vm {\munderbar x}_{k+1} - \initstate \| $ and $\|\vm \lambda\inds(\cdot; \vm {\munderbar x}_{k+1})- \vm \lambda\inds(\cdot; \vm {\munderbar x}_{k})\|_{L_\infty} \leq L_{\lambda} \| \vm {\munderbar x}_{k+1} - \initstate \| $ for all $\initstate \in \Gamma_\alpha$, $k \in \mathbb{N}_0$ (recall that $\vm {\munderbar x}_{k+1} \in \Gamma_\alpha$ due to Lemma \ref{lem:cost_bound}) with which the norm can be bounded as 
	\begin{align} \label{eq:optimal_bound}
		\bigg\| \begin{matrix} 
			\vm x\inds(\cdot; \vm {\munderbar x}_k)- \vm x\inds(\cdot; \vm {\munderbar x}_{k+1})\\ \vm \lambda\inds(\cdot; \vm {\munderbar x}_k) -\vm \lambda\inds(\cdot; \vm {\munderbar x}_{k+1})
		\end{matrix} \bigg \|_{L_\infty} \leq \|\vm x\inds(\cdot; \vm {\munderbar x}_k)- \vm x\inds(\cdot; \vm {\munderbar x}_{k+1})\|_{L_\infty} + \|\vm \lambda\inds(\cdot; \vm {\munderbar x}_k)- \vm \lambda\inds(\cdot; \vm {\munderbar x}_{k+1})\|_{L_\infty} \leq (L_\lambda + L_x) \| \vm {\munderbar x}_{k+1} - \initstate \|.
	\end{align} 
	The norm of the current system state $ \initstate$ can be related to the next state $ \vm {\munderbar x}_{k+1} = \vm x_c^{q_k}(\Delta t; \initstate)$ via the error \eqref{eq:error} 
	\begin{align}
		\| \vm x_c^{q_k}(\Delta t; \initstate)\| = \| \vm x\inds(\Delta t; \initstate) + \Delta \vm x_c^{q_k}(\Delta t; \initstate) \| \geq \|\vm x\inds(\Delta t; \initstate) \| - \|\Delta \vm x_c^{q_k}(\Delta t; \initstate)\| \geq ( \mathrm{e}^{-\hat L \Delta t} - D d)\| \initstate\|\,,
	\end{align}
	where the reverse triangle inequality together with \eqref{eq:state_lower_bound} and \eqref{eq:error_bound} was used.
	If $d \leq \mathrm{e}^{-\hat L \Delta t}/ D$, then $( \mathrm{e}^{-\hat L \Delta t} - D d)>0$ and we can upper bound the current state $\|\initstate \|$ by the next state $\|\vm {\munderbar x}_{k+1}\|$ according to 
	\begin{equation} \label{eq:state_bound}
		\| \initstate\| \leq \frac{1}{\mathrm{e}^{-\hat L \Delta t} - D d} \| \vm {\munderbar x}_{k+1}\|.
	\end{equation}
	Utilizing \eqref{eq:state_bound} to bound the current state $\| \initstate\|$ in \eqref{eq:optimal_bound} and \eqref{eq:suboptimal_bound} via the next state $ \vm {\munderbar x}_{k+1}$ leads to 
	\begin{align}
		\bigg\| \begin{matrix}
			\vm x^q(\cdot;\vm {\munderbar x}_{k})- \vm x\inds(\cdot; \vm {\munderbar x}_{k})\\ \vm \lambda^q(\cdot; \vm {\munderbar x}_{k}) -\vm \lambda\inds(\cdot; \vm {\munderbar x}_{k})
		\end{matrix} \bigg \|_{L_\infty} \leq c_1 \| \vm {\munderbar x}_{k+1} \|\,, \quad \bigg\| \begin{matrix} 
			\vm x\inds(\cdot; \vm {\munderbar x}_{k})- \vm x\inds(\cdot; \vm {\munderbar x}_{k+1})\\ \vm \lambda\inds(\cdot; \vm {\munderbar x}_k) -\vm \lambda\inds(\cdot; \vm {\munderbar x}_{k+1})
		\end{matrix} \bigg \|_{L_\infty} \leq c_2 \| \vm {\munderbar x}_{k+1} \|
	\end{align} 
	with $c_1:= \frac{Ndp}{(1-p)(\mathrm{e}^{-\hat L \Delta t} - D d)}>0 $ and $c_2 := (L_\lambda + \bar L \mathrm{e}^{\hat L T})(1+\frac{1}{\mathrm{e}^{-\hat L \Delta t} - D d})>0$. Inserting these bounds into \eqref{eq:induction_step} results in 
	\begin{align}
		\bigg\| \begin{matrix}
			\delta \vm x^q(\cdot;\vm {\munderbar x}_{k+1})\\ \delta \vm \lambda^q(\cdot;\vm {\munderbar x}_{k+1})
		\end{matrix} \bigg \|_{L_\infty} \leq (1 + p) (c_1 +c_2)p^{q-1} \|\vm {\munderbar x}_{k+1}\|.
	\end{align}
	Utilizing the condition on the iteration number \eqref{eq:th3} with $e:= c_1 + c_2>0$, the bound \eqref{eq:induction_step} eventually becomes 
	\begin{align}
		\bigg\| \begin{matrix}
			\delta \vm x^q(\cdot;\vm {\munderbar x}_{k+1})\\ \delta \vm \lambda^q(\cdot;\vm {\munderbar x}_{k+1})
		\end{matrix} \bigg \|_{L_\infty} \leq d \|\vm {\munderbar x}_{k+1}\|\,,
	\end{align}
	which completes the induction step. Following the lines of Theorem \ref{th:sub_MPC}, this implies the exponential decay of the error, i.e. $\| \Delta \vm x_c\indqk(\cdot;\initstate)\|_{L_\infty}\leq   d\frac{D}{\sqrt{m_J}} \sqrt{p_J^k \bar{\alpha}} $, which in return shows the exponential decay of the cost, i.e. $J\inds(\initstate)\leq p_J^k \bar{\alpha}$.
\end{proof}
Theorem \ref{th:sub_MPC_num_iter} shows that there exists a sufficiently large number of iterations $q_{\mathrm{max}}$ such that the stopping criterion \eqref{eq:stopp_crit} is fulfilled in every MPC iteration, thus, guaranteeing exponential stability of the DMPC scheme. Executing Algorithm \ref{alg:SENSI_org} with a fixed number of iterations per MPC step has the advantage that the computation time on the agent level and the communication steps are kept constant throughout the MPC steps which is critical in real-time applications. 
\begin{remark}
The number of iterations $q_k$ for satisfying the stopping criterion \eqref{eq:stopp_crit} will be lower in practice since \eqref{eq:th3} is a worst-case estimate for the required number of iterations. Typically, the required number of iterations is highest in the first MPC step $k=0$ due to the "cold" start with some initial trajectories and decreases with the progression of the MPC scheme to the warm start \eqref{eq:warm_start}, see also the numerical example in Section \ref{subsec:benchmark}.
\end{remark}

%

\section{Numerical Evaluation}
\label{sec:num_eval}
Two numerical examples are utilized to demonstrate the performance of Algorithm \ref{alg:SENSI_org} and its theoretical properties. At first, a nonlinear system of three agents is utilized to compare the distributed to the central solution and to verify the linear convergence property. In addition, a comparison to the ADMM algorithm\cite{Boyd} is presented. The next example concerns the distributed end region and the effect of the suboptimal predicted trajectories on the closed-loop behavior. Algorithm~\ref{alg:SENSI_org} is implemented in C++ within the modular DMPC framework GRAMPC-D\cite{Burk2,Pierer} in which the local OCPs \eqref{eq:local_ocp} are solved with the toolbox GRAMPC\cite{Englert} via the projected gradient method.  
\subsection{Benchmark system}
\label{subsec:benchmark}
In this section, a typical benchmark system which is often used in nonlinear DMPC \cite{Dunbar,Bestler,Zhou.2017,Zhou.2017b,Zhou.2019} is considered for the application of the sensitivity-based DMPC scheme. The dynamics of the system describing a bipedal locomotor experiment are given by the following nonlinear differential equations  
\begin{align} 
	\ddot \theta_1 &= 0.1(1 -5.25 \theta_1^2)\dot \theta_1 - \theta_1 + u_1  \nonumber\\
	\ddot \theta_2 &= 0.001(1 - 6070\theta_2^2)\dot \theta_2 - 4\theta_2 + 0.057\theta_1 \dot \theta_1 + 0.1(\dot \theta_2 - \dot \theta_3) + u_2 \label{eq:VDP} \\
	\ddot \theta_3 &= 0.001(1 - 192\theta_3^2 ) \dot\theta_3 - 4\theta_3+ 0.057\theta_1 \dot \theta_1 + 0.1(\dot \theta_3 - \dot \theta_2) +u_3 \nonumber
\end{align}
with the states $\vm x_i = [\theta_i, \dot \theta_i]$ and controls $u_i $ of each agent $ i \in \mathcal{V}=\{1,2,3\}$ for which the system \eqref{eq:VDP} can be represented in the neighbor-affine form \eqref{eq:agent_dyn}. The coupling structure of the system follows as $\mathcal{N}_1^\rightarrow=\{2,3\}$, $\mathcal{N}_2^\rightarrow=\{3\}$, $\mathcal{N}_3^\rightarrow=\{2\}$ and $\mathcal{N}_1^\leftarrow=\{\}$, $\mathcal{N}_2^\leftarrow=\{1,3\}$, $\mathcal{N}_3^\leftarrow=\{1,2\}$. 
The controls are constrained to the set $ u_i(t) \in [-1,1]$, $\forall \agents$ and the control task is to drive the system from its stable limit cycle to the origin. To this end, the quadratic cost functions 
$
l_i(\vm x_i, u_i) = \vm x_i\trans \vm Q_i \vm x_i + R_i u_i^2,\, V_i(\vm x_i) = \vm x_i\trans \vm P_i \vm x_i
$
with the weighting matrices $ \vm Q_i = \mathrm{diag}(30,30)$ and $ R_i = 0.1$ are employed for each agent. We follow the approach of Section \ref{subsec:DMPC_stab_opt} and linearize the system around the origin to obtain the linear system with $\vm A = \partial_{\vm x} \vm f(\vm x, \vm u)$ and $\vm B = \partial_{\vm u} \vm f(\vm x, \vm u)$ which is linearly controllable but unstable. The optimization procedure \eqref{eq:min_distCost} with $\gamma =1.2$ is utilized offline to find the terminal weights
\begin{equation}
	\vm P_1 = \begin{bmatrix}
		37.4 & 2.0 \\ 2.0 & 2.2
	\end{bmatrix}\,, \quad \vm P_2 = \vm P_3 = \begin{bmatrix}
		38.8 & 1.7 \\ 1.7 & 2.2
	\end{bmatrix}
\end{equation}
and structured feedback controller $ \vm K_{11} = [-16.3,-18.3]$, $ \vm K_{12} = \vm K_{13} = \vm 0$ and $\vm K_{22} = \vm K_{33} =[ -13.8,\, -15.3]$, $ \vm K_{23} = K_{32} = [0.02,\, 0.04]$ and $\vm K_{21} = \vm K_{31} = \vm 0$.
Both the conditions that the linear controller satisfies the input constraints, i.e. $ \munderbar{u}_i(t) \in [-1,\,1]$, and that the CLF inequality \eqref{eq:CLF_inequality} are fulfilled in the terminal region $\Omega_\beta = \{\vm x\trans \vm P \vm x \leq 0.9\}$. Finally, the prediction horizon is set to $T=3.0 \,\unit{\second}$ and the sampling time to $\Delta t = 0.05 \unit{\second}$, respectively.
The initial values for the (D)MPC simulation illustrated in Figure \ref{fig:simulVDP} are chosen as $\theta_{1,0} = 0.7$, $\theta_{2,0} = 0.28$, $\theta_{3,0} = -0.61$ and $\dot \theta_{i,0} = 0 $ according to reference\cite{Dunbar}, for all $\agents$ and the system is simulated for a total of $6 \unit{\second}$. The top two figures show a comparison of the states and controls for the centralized optimal solution and the solution obtained by Algorithm \ref{alg:SENSI_org} with $d=0.1$ in the stopping criterion \eqref{eq:stopp_crit}. Furthermore, the lower left figure shows the exponential decay of the cost function $J(\initstate) = \sum_{\agents} J_i(\vm u_i; \initstate)$ in each MPC step which is in line with Theorems \ref{th:centralMPC} and \ref{th:sub_MPC}. The required iterations $q_k$ such that the stopping criterion \eqref{eq:stopp_crit} is fulfilled are depicted in the lower right figure for different values of $d$. Clearly, a tighter stopping bound leads to more required iterations. As indicated by Theorem \ref{th:sub_MPC_num_iter} the required iterations are highest in the initial MPC step due to the large initial optimization error and then converge to a stationary value. This is especially prevalent in the case of $d=0.1$ where four algorithm iterations are required to overcome the initial optimization error and then subsequently drop to two iterations for the remaining MPC steps due to the warm-start \eqref{eq:warm_start}.
The overall communication effort can be characterized by the number of trajectories sent between an agent $i$ and its neighbor $\neighs$ in each MPC sampling step. Consequently, Algorithm \ref{alg:SENSI_org} requires a total number of $ q_k n_i (|\mathcal{N}_i^\leftarrow| + |\mathcal{N}_i|)$ trajectories to be communicated by each agent. In this example, the first MPC step for $d=0.1$ thus requires $80$ trajectories to be exchanged in the communication network and is halved to $40$ trajectories in the following steps. In practice, the trajectories are transmitted in discretized form and the actual data amount depends on the number of discretization points (21 in this example). The major advantage however is that only one local communication step is needed in each iteration $q$ which greatly reduces the overall communication time in the network.\cite{Pierer}
\begin{figure}[tb] 
	\centering
	\includegraphics{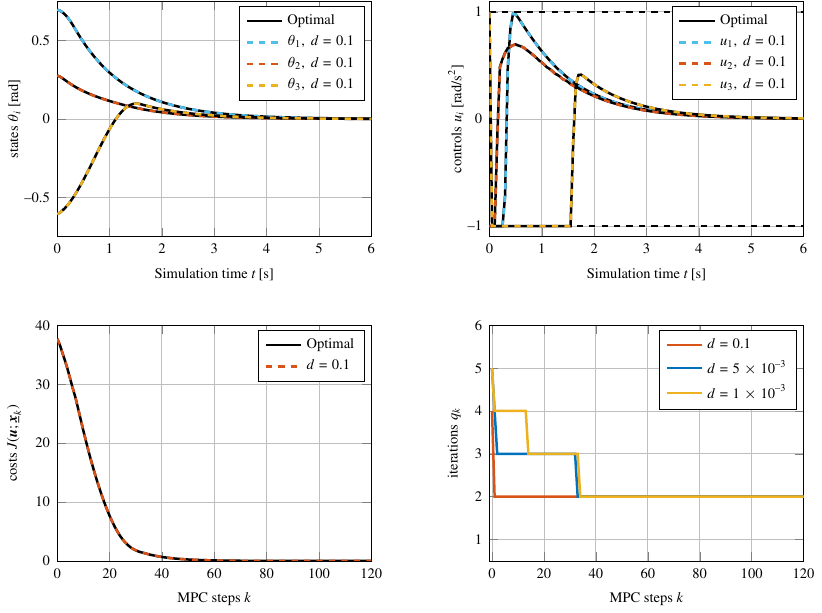}
	\caption{Closed-loop optimal and distributed solution of System \eqref{eq:VDP}. } \label{fig:simulVDP}
\end{figure}
Figure \ref{fig:conv_VDP} shows the evolution of the envelope and mean norm between optimal and suboptimal trajectories in each step $q$ of Algorithm \ref{alg:SENSI_org} in the first MPC step for different initial conditions. According to Lemma \ref{lem:lin_conv}, linear convergence in each step is guaranteed as long as the prediction horizon is chosen sufficiently small which is confirmed by this example system. It is also apparent that the adjoint states are the limiting factor for convergence as their progress is slower than that of the states which also can be seen in the proof of Lemma~\ref{lem:lin_conv}. This is mainly due to the fact that initialization is further away from the optimal solution than in the case of the states and that the "negotiation" between agents takes place via the adjoint states. 
\begin{figure}[tb] 
	\centering
	\includegraphics{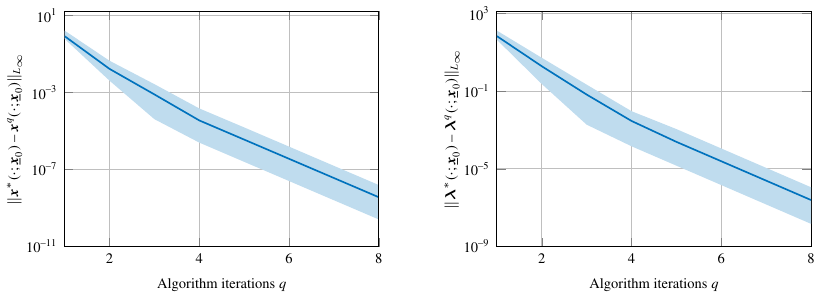}
	\caption{Envelope over all sampled trajectories together with the median convergence curve of Algorithm \ref{alg:SENSI_org} for random initial conditions in the first DMPC step $k=0$. }\label{fig:conv_VDP}
\end{figure}
In addition, a comparison with the popular ADMM algorithm \cite{Boyd} is given. In particular, we utilize the ADMM implementation of GRAMPC-D\cite{Burk2} with penalty parameter adaption and employ the stabilizing stopping criterion from \cite{Bestler} which is similar to \eqref{eq:stopp_crit} with $d=5\times 10^{-3}$ and yields approximately the same control performance. The goal of stabilizing the coupled oscillators \eqref{eq:VDP} remains the same. However, to ensure maximum comparability, we use the same cost matrices and initial values as in \cite{Bestler} to arrive at an identical setup. The convergence behavior is evaluated by observing the normalized cost progression w.r.t. the number of algorithm iterations in the first MPC step $k=0$ which is visualized in the left of Figure \ref{fig:Compare_ADMM}. The sensitivity-based algorithm takes about 4 iterations to convergence while the ADMM algorithm requires approximately 30 iterations to converge towards the central solution. Furthermore, the number algorithm iterations $q_k$ to achieve a stabilizing solution is depicted on the right of Figure \ref{fig:Compare_ADMM}. Clearly, the required iterations $q_k$ in each MPC-step of the proposed algorithm are lower compared to ADMM underpinning the promising application possibilities of the sensitivity-based algorithm to DMPC.
\begin{figure}[tb] 
	\centering
	\includegraphics{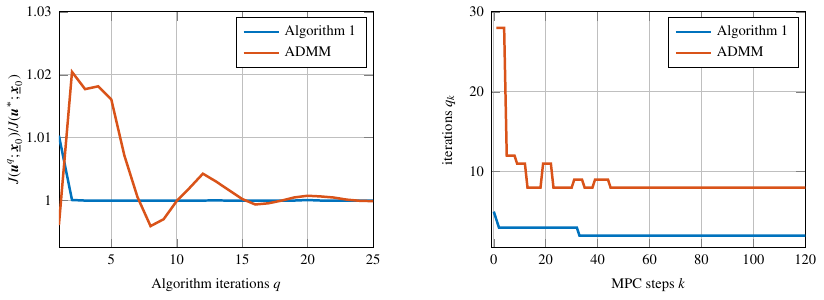}
	\caption{Comparison of sensitivity-based and ADMM-based DMPC in terms of the convergence behavior in the first DMPC step (left) as well as the required iterations for a stabilizing solution (right).} \label{fig:Compare_ADMM}
\end{figure}
\subsection{Distributed end region}
The next example concerns the distributed end region and investigates how the suboptimality of the predicted trajectories affect the region of attraction. As an example we use a similar system as in the work\cite{Chen}
\begin{equation} \label{eq:dist_end_system}
	\dot x_i =  (\mu_i + (1- \mu_i)x_i)u_i + \sum_{\send} \epsilon_{ij} x_j\,,
\end{equation}
where a number of $N=2$ agents are considered which are coupled bi-directionally resulting in $\mathcal{N}_1 = \{2\}$ and $\mathcal{N}_2 = \{1\}$. The controls are constrained to the set $u_i \in [-2,\,2]$ for all $ \agents$.
Again the quadratic functions \eqref{eq:quadratic_costs} are chosen with the weighting matrices $ \vm Q = \mathrm{diag}(10,10)$ and $ \vm R = \mathrm{diag}(1,1)$. For now set $ \epsilon = \epsilon_{12} = \epsilon_{21} = 2$ and $\mu_1 = \mu_2 = \mu = 0.5$. The optimization problem~\eqref{eq:min_distCost} is solved with $\gamma = 1.1$ with and without the structural constraint \eqref{eq:distCost_structure} to obtain a separable $V(\vm x)= \vm x\trans \vm P_d \vm x$ and non-separable terminal cost function $V(\vm x)= \vm x\trans \vm P_c \vm x$. Consequently, the invariant end region $\Omega_{\beta} = \{ \vm x\trans \vm P_{c/d} \vm x \leq 0.9\}$ is obtained for both terminal cost functions since the dynamics are identical for both agents. The left plot of Figure \ref{fig:dist_end} visualizes the two terminal regions. Clearly, the size of $\Omega_{\beta}$ is reduced in the distributed setting and the separable terminal region is contained within the non-separable end region. The right plot shows the change in area of the terminal region which is proportional to $\det(\vm P^{-1})$ with respect to an increasing coupling strength $\epsilon$ for different $\mu$. The intuitive result that the separable terminal region is reduced compared to the non-separable one with an increase in coupling strength is visible. In addition, an increased nonlinearity (for $\mu = 1$ system \eqref{eq:dist_end_system} is linear) worsens this effect and reduces the terminal region even further. 
\begin{figure}[b]
	\centering
	\includegraphics{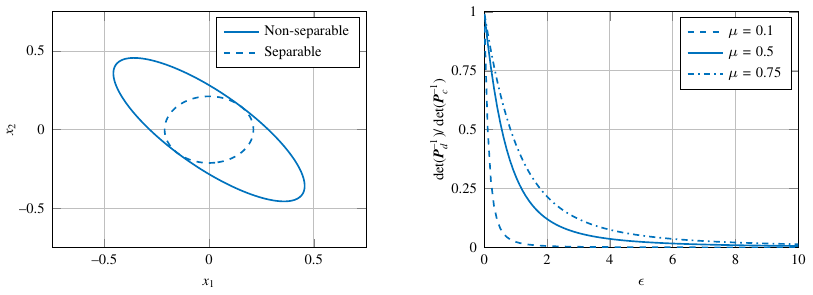}
	\caption{Comparison of the end regions following from separable and non-separable terminal cost functions.} \label{fig:dist_end} 
\end{figure}
The unstable system \eqref{eq:dist_end_system} is now stabilized via the DMPC control scheme for different initial conditions. The coupling strengths are set to $\epsilon_{12} =0.5$ and $\epsilon_{21} = 2.0$, $\mu_1 = 1.0$ and $\mu_2 = 0.5$, the prediction horizon to $T = 0.5\, \unit{\second}$ and the sampling to $\Delta t = 0.05\,\unit{\second}$. The terminal cost is again designed via the procedure of Section \ref{subsec:DMPC_stab_opt} to obtain the end region  $\Omega_{\beta} = \{\vm x\trans \vm P \vm x \leq 1.05\}$. Figure~\ref{fig:dist_end_traj} shows the optimal predictions obtained by solving OCP \eqref{eq:central_ocp}, the suboptimal prediction of Algorithm \ref{alg:SENSI_org} after $q=1$ iteration in which OCP \eqref{eq:local_ocp} is solved once and the resulting closed-loop trajectory of the suboptimal DMPC controller with $q_{\mathrm{max}} = 1$ as well as the terminal region $\Omega_{\beta}$. All initial conditions where chosen to be located within the region of attraction $\Gamma_{\alpha}$ since all endpoints of the predicted trajectories reach the terminal region, i.e. $\vm x\inds(T; \vm x_0) \in \Omega_\beta$. Thus, with the chosen initial conditions the optimal MPC controller is guaranteed to be stable according to Theorem \ref{th:centralMPC}. This, however, is not necessarily the case for the suboptimal distributed controller.

Note the trajectories for the upper left initial condition at $ \vm x_0 = [-1.3\, 1.4]\trans$. They visualize the difference between the original domain of attraction $\Gamma_{\alpha}$ and the reduced domain of attraction $\Gamma_{\bar \alpha}$ as the next sampling point of the suboptimal closed-loop trajectory $\vm x_c(\Delta t;  \munderbar{\vm x}_0)$ does not lie within the region of attraction $\Gamma_{\alpha}$ such that Theorems \ref{th:sub_MPC} and \ref{th:sub_MPC_num_iter} hold. This can be seen from the fact that the endpoint of the optimal predicted trajectory at the next sampling point $\munderbar{\vm x}_1 = \vm x_c(\Delta t;  \munderbar{\vm x}_0)$ does not lie within the terminal region, i.e. $ \vm x\inds(T; \munderbar{\vm x}_1) \notin \Omega_{\beta}$. 

Thus, the DMPC controller with this number of algorithm iterations is not guaranteed to be stable for this particular initial condition. This could be circumvented by either performing more iterations or choosing an initial condition closer to the origin such that it is located within $\Gamma_{\bar \alpha}$. However, the scheme is still robust enough to stabilize the system in this particular case. Moreover, the dependency on the sampling time can also be seen in this particular trajectory as a higher sampling time would take the next sampling point $\vm x_c(\Delta t;  \munderbar{\vm x}_0)$ farther away from the region of attraction and thus possibly destabilize the scheme.
\begin{figure}[tb]
	\centering
	\includegraphics{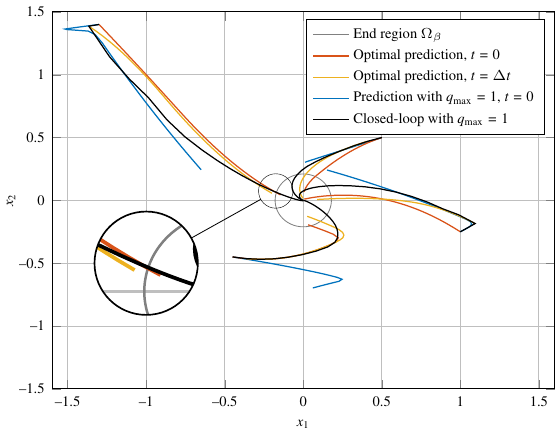}
	\caption{Predicted and closed-loop trajectories of the controlled system \eqref{eq:dist_end_system} for different initial conditions.}\label{fig:dist_end_traj}
\end{figure}

%
\section{Conclusion}
\label{sec:concl}
This paper presented a DMPC scheme for continuous-time nonlinear systems which relies on sensitivities to cooperatively solve the underlying OCP in each MPC sampling step in a distributed fashion. The agents are dynamically coupled in a neighbor-affine form which has the consequence that the sensitivities can be evaluated locally. The algorithm only requires local computations and one neighbor-to-neighbor communication step per iteration and thus constitutes a scalable and fully distributed DMPC scheme. Furthermore, it was shown that linear convergence can be guaranteed as long as the prediction horizon is sufficiently small which represents a compromise between the domain of attraction and convergence speed of the algorithm. The computation and communication load is limited by either a contracting stopping criterion or a fixed number of algorithm iterations. For both scenarios exponential stability of the DMPC controller is shown. Numerical evaluations have demonstrated the effectiveness of the presented scheme as only a few iterations per MPC step are necessary to achieve a nearly optimal performance.

Future research concerns the experimental validation of the presented DMPC scheme. In addition, convergence of the algorithm will be analyzed in the presence of packet loss and in the asynchronous setting to further reduce the execution time. 
%
%
%
%
\bibliography{Pierer_24.bib}

\appendix
\bmsection{Computation of Sensitivities}
\label{appA:sens_calc}
In order to derive the first-order sensitivities for OCP \eqref{eq:local_ocp}, the OCP of an individual neighbor $\rec$ of the central OCP \eqref{eq:central_ocp} from the perspective of the agent $\agents$ is considered
\begin{subequations}\label{eq:deriv_ocp}
	\begin{align}
		\min_{ \vm u_{j}} \quad & J_j(\vm x_j, \vm u_j, \Nj{x};\initstate) := V_j(\vm x_j(T)) + \int_0^T  l_j( \vm x_j, \vm u_j, \Nj{x})\, \dd \tau
		\\  ~\stt \quad&\vm{\dot x}_j = \vm f_j(\vm x_j, \vm u_j,\Nj{x})\,, 
		\quad \vm x_j(0) =\vm {\munderbar x}_{j,k}\\
		&\vm u_j(\tau) \in \mathbb{U}_j\,, \quad \timetau\,.
	\end{align}
\end{subequations}
Now the first-order sensitivity is formulated as the Gâteaux derivative of OCP \eqref{eq:deriv_ocp}
\begin{equation} \label{eq:Gateax_deriv}
	\delta J_j(\vm x_j\indqp, \vm u_j \indqp, \Nj{x}\indqp;\initstate)(\delta \vm x_i) =\frac{\mathrm{d}J_j(\vm x_j\indqp, \vm u_j\indqp, \Nj{x}\indqp + \epsilon \delta \vm x_{ji};\initstate)}{\mathrm{d}\epsilon}\bigg|_{\epsilon = 0}
\end{equation}
at some point $(\vm x_j\indqp, \vm u_j \indqp, \Nj{x}\indqp)$ w.r.t. the admissible direction $\delta \vm x_{ji} = [\vm 0\trans \dots \vm 0\trans \, \delta \vm x_i\trans\, \vm 0\trans \dots \vm 0\trans]\trans \in \mathds{R}^{p_j}$ where $\delta \vm x_i=\vm x_i - \vm x_i\indqp$ shows up at the position in $\delta \vm x_{ji}$ corresponding to $\vm x_i$ in $\Nj{x}$. 
Adjoining the dynamics to the
cost via (time-dependent) Lagrange multipliers $\vm \lambda_j
\in \mathds{R}^{n_j}$ results in
\begin{align} 
	\delta J_j(\vm x_j\indqp, \vm u_j\indqp, \Nj{x}\indqp;\initstate )(\delta \vm{x}_i) &=\int_0^T \frac{\mathrm{d}}{\mathrm{d} \epsilon}l_j(\vm x_j\indqp, \vm	u_j\indqp, \Nj{x}\indqp +  \epsilon \delta \vm{x}_{ji}) +\frac{\mathrm{d}}{\mathrm{d}\epsilon} \big((\vm \lambda_j\indqp)\trans (\,\vm f_j(\vm x_j\indqp, \vm u_j\indqp, \Nj{x}\indqp + \epsilon \delta \vm{x}_{ji}) - \vm{\dot{x}}_j\indqp)\big)\, \dd t\,\bigg|_{\epsilon = 0}  \nonumber \\
	& = \int_0^T\partial_{\vm {x}_{ji}} \big(l_j(\vm x_j\indqp, \vm u_j\indqp, \Nj{x}\indqp) +\vm f_j(\vm x_j\indqp, \vm u_j\indqp, \Nj{x}\indqp)\trans \vm \lambda_j\indqp\big)\trans\,\delta \vm{x}_{ji}\, \dd \tau \label{eq:Gateaux_deriv}.
\end{align} 
Considering the particular formulation of dynamics
\eqref{eq:agent_dyn} and cost terms \eqref{eq:agent_cost} in neighbor-affine form, the
Gâteaux derivative further simplifies to  
\begin{align} \label{eq:def_Gateaux_x}
	\delta J_j(\vm x_j\indqp, \vm u_j\indqp, \Nj{x}\indqp;\initstate)(\delta \vm x_i) = \int_0^T \big(\partialx l_{ji}(\vm x_j\indqp, \vm x_i\indqp) + \partialx\,\vm f_{ji}(\vm x_j\indqp, \vm x_i\indqp)\trans \vm \lambda_j\indqp\big)\trans\,\delta \vm x_i \, \dd \tau.
\end{align}
In OCP \eqref{eq:local_ocp}, the sensitivities are calculated recursively for the already augmented cost functional. Recalling the definition of $\delta \vm x_i= \vm x_i - \vm x_i\indqp$, it is clear that \eqref{eq:def_Gateaux_x} is linear w.r.t. $\vm x_i$ and can be incorporated as a linear term into the local cost function term $l_{ii}(\vm x_i, \vm u_i)$ in
\eqref{eq:agent_cost}. Thus, a repeated application of the Gâteaux derivative \eqref{eq:Gateax_deriv} still results in \eqref{eq:def_Gateaux_x}, i.e. the sensitivity \eqref{eq:recursive_Gateaux} $\delta \bar J_j(\vm u_j\indqp;\initstate )$ of the augmented cost function in \eqref{eq:local_costFunction}.
\bmsection{Proof of Lemma 1}
\begin{proof}
	The following considerations require several Lipschitz estimates. Based on the continuous differentiability of the dynamics \eqref{eq:pred_traj}, there exists a finite (local) Lipschitz constant $L_f<\infty$ for some $r_x>0$ such that
	\begin{align} \label{eq:Lipschitzf}
		\|\vm{ \hat f}(\vm x, \vm u,\vm{\hat x}) - \vm{ \hat f}(\vm y, \vm v,\vm{\hat y})\|&\leq L_f(\| \vm x-\vm y\| + \| \vm u-\vm v\| + \|\vm{\hat x}- \vm{\hat y}\|) 
	\end{align}
	for all $\vm x, \vm y,\vm{\hat x},\vm{\hat y} \in \Gamma_{\alpha}^{r_x}$ and $\vm u, \vm v \in \mathbb{U}$, where $\Gamma_{\alpha}^{r_x} :=  \bigcup_{\vm x_k \in \Gamma_{\alpha}} \mathcal{B}(\vm x_k, r_x)$ is the $r_x$-neighborhood to the (compact) domain of attraction~$\Gamma_{\alpha}$.  Similarly, consider the stacked adjoint dynamics \eqref{eq:adjoint_dyn} in each iteration $q$ of Algorithm \ref{alg:SENSI_org}, i.e.
	\begin{equation}\label{eq:iter_adjoint_dyn}
		\vm{ \dot \lambda}\indq(\tau) = -\partial_{\vm x} \sum_{\agents}	H_i(\vm x_i, \vm u_i, \vm \lambda_i)=: \vm G_d(\vm x\indq, \vm u\indq, \vm \lambda\indq; \vm{\bar x}\indqp,\vm{\bar \lambda}\indqp) \,, \quad  \vm \lambda\indq(T) = \partial_{\vm x} V(\vm x\indq(T))\,,
	\end{equation}
	where the notation $\vm {\bar x}= [\vm x_i\trans\,\vm x_{\mathcal{N}_i}\trans]_{\agents} \in \mathds{R}^{p_x}$, $ p_x = \sum_{\agents}(n_i+ \sum_{\neighs} n_j)$ and  $\vm {\bar \lambda} = [\vm \lambda_{\mathcal{N}_i^\rightarrow}]_{\agents} \in \mathds{R}^{p_{\lambda}}$, $p_{\lambda} = \sum_{\agents} \sum_{\rec} n_j$ explicitly captures the dependency on the trajectories of iteration $q-1$. Due to the assumed differentiability of all dynamics and cost functions, there exists a finite Lipschitz constant $L_G<\infty$ for some $r_\lambda>0$ such that
	\begin{align}
		\|\vm G_d(\vm x, \vm u, \vm \lambda; \vm{\bar x}, \vm{\bar \lambda}) - \vm G_d(\vm y, \vm v, \vm \mu;  \vm{\bar y}, \vm{\bar \mu})\|&\leq L_G(\| \vm x-\vm y\| + \| \vm u-\vm v\| + \| \vm \lambda - \vm \mu\| +\|\vm{\bar x} -\vm{\bar y}\| + \|\vm{\bar \lambda} -\vm{\bar \mu}\|)  \label{eq:LipschitzG}
	\end{align}
	for all for all $\vm x, \vm y,\vm{\bar x},\vm{\bar y} \in \Gamma_{\alpha}^{r_x}$ and $\vm u, \vm v \in \mathbb{U}$ as well as $\vm \lambda, \vm \mu,\vm{\bar \lambda},\vm{\bar \mu} \in \mathcal{S}^{r_\lambda}$, where and $\mathcal{S}^{r_\lambda}$ is the $r_\lambda$-neighborhood to the compact set $\mathcal{S} = \{\partial_{\vm x} V(\vm x)| \vm x \in \Gamma_{\alpha}^{r_x}\}$. 
	
	At first it is shown that the iterates are bounded for each $\initstate \in \Gamma_{\alpha}$ for a sufficiently short prediction horizon $T$, i.e. $ \vm x\indq(\tau) \in \Gamma_{\alpha}^{r_x}$ and  $\vm \lambda\indq(\tau) \in \mathcal{S}^{r_\lambda}$, $\timetau$ for $q=1,2,\dots,q_{\mathrm{max}}$. We proceed by induction. To this end, assume that $\vm x\indqp(\tau)\in \Gamma_{\alpha}^{r_x}$ and $\vm \lambda\indqp(\tau) \in \mathcal{S}^{r_\lambda}$, $\timetau$ and consider the integral form of the dynamics \eqref{eq:pred_traj} for $\initstate \in \Gamma_{\alpha}$. By adding and subtracting $\vm{ \hat f}(\initstate, \vm 0, \vm{\munderbar{\hat x}}_k)$, with $\vm{\munderbar{\hat x}}_k =[[\vm{\munderbar{x}}_{j,k}]_{\send}]_{\agents}$ as well as using the Lipschitz property \eqref{eq:Lipschitzf}, we get (omitting time arguments)
	\begin{align} 
		\|\vm x\indq(\tau;\initstate)- \initstate\| &\leq \int_0^\tau \|\vm{ \hat f}(\vm x\indq, \vm u\indq,\vm {\hat x}\indqp) - \vm{ \hat f}(\initstate, \vm 0, \vm{\munderbar{\hat x}}_k) + \vm{ \hat f}(\initstate, \vm 0, \vm{\munderbar{\hat x}}_k)\|\, \dd s \nonumber\\
		&\leq  \int_0^\tau L_f(\| \vm x\indq - \initstate \| + \| \vm u\indq \| + \| \vm {\hat x}\indqp- \vm{\munderbar{\hat x}}_k\|) + \|\vm{ \hat f}(\initstate, \vm 0, \vm{\munderbar{\hat x}}_k)\| \,\dd s\,. \label{eq:bound_x}
	\end{align}
	The norm $\| \hat{\vm x}\indqp -\vm{\munderbar{\hat x}}_k\|$ in \eqref{eq:bound_x} at time $s$ can be bounded further by realizing that $\vm {\hat x}\indqp(s)$, $ \vm{\munderbar{\hat x}}_k(s) \in \mathds{R}^p$ exclusively consist of elements of $\vm {x}\indqp$ and  $\initstate$
	\begin{align}\label{eq:bound_init_hat}
		\| \hat{\vm x}\indqp -\vm{\munderbar{\hat x}}_k\| \leq \sqrt{ p} \|\hat{\vm x}\indqp -\vm{\munderbar{\hat x}}_k\|_\infty = \sqrt{ p} \| \vm {x}\indqp - \initstate\|_\infty \leq \sqrt{ p} \|\vm {x}\indqp - \initstate\|\,.
	\end{align}
	By Gronwall’s inequality, the bound $ \| \vm u\indq(\tau)\| \leq r_u:=\max_{\vm u \in \mathbb{U}} \| \vm u\|<\infty$,  and the fact that $\|\vm{ \hat f}(\initstate, \vm 0, \vm{\munderbar{\hat x}}_k)\|\leq h_f < \infty $ due to the continuity of $\vm{ \hat f}$, it follows that
	\begin{align}
		\|\vm x\indq(\tau;\initstate)- \initstate\| &\leq \tau \mathrm{e}^{L_f\tau}( L_f(r_u + \sqrt{p} r_x) + h_f)\,, \quad \timetau
	\end{align}
	Thus, by choosing $T<T_f$ with $T_f$ satisfying $T_f \mathrm{e}^{L_f T_f}( L_f(r_u + \sqrt{p} r_x) + h_f) = r_x$, the state trajectories are bounded, i.e. $\vm x_i\indq(\tau) \in \Gamma_{\alpha}^{r_{x}}$, $\timetau$.
	
	The boundedness of $\vm \lambda\indq(\tau)$ can be shown in similar fashion by considering the integral form of \eqref{eq:iter_adjoint_dyn} in reverse time and the notation $y_r(\tau) := y(T - \tau) $ for some trajectory $y(\tau),\, [0,\, T_x]$. By adding and subtracting $\vm G_d(\initstate, \vm 0, \vm \lambda\indq(T); \vm{\munderbar{\bar x}}_k,\vm{{\bar \lambda}}\indq(T))$, utilizing the Lipschitz property \eqref{eq:LipschitzG} and the fact that $\vm x\indq(\tau) \in \Gamma_{\alpha}^{r_x}$, one gets
	\begin{align}
		\| \vm \lambda_r\indq(\tau; \initstate) - \vm \lambda_r\indq(0; \initstate)\| \leq \int_0^\tau \|\vm G_d(\vm x_r\indq, \vm u_r\indq, \vm \lambda\indq; \vm{\bar x}_r\indqp,\vm{\bar \lambda}_r\indqp)\|\, \dd s \leq \tau \mathrm{e}^{L_G\tau}(L_G(r_x + r_u + r_x\sqrt{p_x} + r_x\sqrt{p_\lambda}) + h_G)
	\end{align}
	where the bounds similar to \eqref{eq:bound_init_hat}
	\begin{align}
		\| \vm {\bar x}\indqp - \vm{\munderbar{\bar x}}_k\| \leq \sqrt{ p_x} \| \vm {x}\indqp -\vm{\munderbar{x}}_k\|\,, \quad \| \vm {\bar \lambda}\indqp -\vm{{\bar \lambda}}\indq(T)\| \leq \sqrt{ p_\lambda} \| \vm {\lambda}\indqp -\vm{{\lambda}}\indq(T)\|\,,
	\end{align}
	and $\|\vm G_d(\initstate, \vm 0, \vm \lambda\indq(T); \vm{\munderbar{\bar x}}_k,\vm{{\bar \lambda}}(T)\| \leq h_G$ were used. Thus, by choosing $T < \min\{T_x,T_\lambda\}$ with $T_\lambda$ satisfying $T_\lambda\mathrm{e}^{L_GT_\lambda}(L_G(r_x + r_u + r_x\sqrt{p_x} + r_x\sqrt{p_\lambda}) + h_G) =r_\lambda$, the adjoint state trajectory is bounded, i.e. $\vm \lambda\indq(\tau) \in \mathcal{S}^{r_\lambda}$. Furthermore, note that for $q=0$ either $\vm x^0(\tau;\vm{\munderbar x}_0) = \vm x_0 \in \Gamma_\alpha^{r_x}$, $\vm \lambda^0(\tau)= \vm \lambda_T \in \mathcal{S}^{r_\lambda}$ in the first MPC step $k=0$ or for $k>0$, $\vm x^0(\tau;\vm{\munderbar x}_k) = \vm x^{q_k}(\tau;\vm{\munderbar x}_{k-1}) \in \Gamma_\alpha^{r_x} $ and  $\vm \lambda^0(\tau;\vm{\munderbar x}_k) = \vm \lambda^{q_k}(\tau;\vm{\munderbar x}_{k-1}) \in \mathcal{S}^{r_\lambda}$ by \eqref{eq:warm_start}.
	This shows that for $T < \min\{T_x,T_\lambda\}$ all iterates stay within their respective sets $\Gamma_{\alpha}^{r_x}$ and $\mathcal{S}^{r_\lambda}$ in each MPC step and are therefore bounded.
	
	Next, the boundedness of the iterates is used to establish linear convergence of Algorithm \ref{alg:SENSI_org}. To this end, the difference between the optimal solution and the solution in each step $q$ of Algorithm \ref{alg:SENSI_org} needs to be characterized. To this end, define the errors $\Delta \vm x\indq(\tau; \initstate) := \vm x\inds(\tau; \initstate) - \vm x\indq(\tau; \initstate)$, $\Delta \vm u\indq(\tau; \initstate) := \vm u\inds(\tau; \initstate) - \vm u\indq(\tau; \initstate)$ and $\Delta \vm \lambda\indq(\tau; \initstate) := \vm \lambda\inds(\tau; \initstate) - \vm \lambda\indq(\tau; \initstate)$, $\timetau$. We derive a bound on $\Delta \vm x\indq(\tau;\initstate)$ by considering the difference of the integral form of the dynamics \eqref{eq:pred_traj} and \eqref{eq:opt_traj} at some step $q$ 
	\begin{align} \label{eq:bound_Delta_x}
		\|\Delta \vm x\indq(\tau; \initstate) \| &\leq \int_0^\tau \| \vm{ \hat f}(\vm x\inds, \vm u\inds, \hat{\vm x}\inds) - \vm{ \hat f}(\vm x\indq, \vm u\indq, \vm {\hat x}\indqp)\|\, \dd s \leq  L_f \int_0^\tau \|\Delta \vm x\indq\| + \|\Delta \vm u\indq \| + \| \hat{\vm x}\inds - \hat{\vm x}\indqp \|\, \dd s 
	\end{align}
	where again the notation $\vm{\hat x} = [\Ni{x}]_{\agents} \in \mathds{R}^{p}$ is used to achieve structural equivalence between $\vm f(\vm  x\inds, \vm u \inds)$ and $\vm{ \hat f}(\vm x\indq, \vm u\indq, \vm {\hat x}\indq)$.
	
	Note that by assumption any solution of the differentiable dynamics \eqref{eq:centr_dyn} is bounded for bounded controls, i.e. there exists a compact set $\mathbb{X} \subset \mathds 
	R^n$ such that  $\vm x\inds(\tau) \in \mathbb{X}$, $\timetau$ for all $ \vm u(t) \in \mathbb{U}$ and $\initstate \in \Gamma_{\alpha}$. In addition, by choosing $T < \min\{T_x,T_\lambda\}$ the boundedness of any solution of \eqref{eq:pred_traj} at a given iteration $q$ is ensured, i.e.  $\vm x\indq(\tau), \vm x\indqp(\tau)\in \Gamma_\alpha^{r_x}$ such that the Lipschitz estimate \eqref{eq:Lipschitzf} is applicable.
	The norm $\|\Delta \vm u\indq \| = \|\Delta \vm u\indq(\tau; \initstate)\|$ in \eqref{eq:bound_Delta_x}, $\timetau$ is expressed by the feedback laws \eqref{eq:MPC_control_law} and \eqref{eq:DMPC_control_law} 
	\begin{align}
		\int_0^\tau \|\Delta \vm u\indq(s; \initstate)\|\, \dd s = \!\!\int_0^\tau\|\vm \kappa(\vm x\inds(s; \initstate); \vm x\inds,\vm \lambda\inds, \initstate ) - \vm \kappa(\vm x\indq(s; \initstate); \vm x\indqp,\vm \lambda\indqp, \initstate )\|\, \dd s \leq L_\kappa \!\!\int_0^\tau \|\Delta \vm x\indq\| +  \| \Delta \vm x \indqp\| + \|\Delta \vm \lambda\indqp\|\, \dd s \label{eq:bound_Delta_u}
	\end{align}
	where the relation $\vm \kappa\inds(\vm x^*(\tau;\initstate ); \initstate) =  \vm \kappa(\vm x\inds(\tau; \initstate); \vm x\inds,\vm \lambda\inds, \initstate )$ is utilized. Note that the states $(\vm x\inds, \vm x\indq,\vm x \indqp)$ and adjoint states $(\vm \lambda\inds, \vm \lambda\indq,\vm \lambda \indqp)$ are defined on compact sets. Hence, the Lipschitz property of the feedback law in Assumption \ref{ass:cost_Lipschitz} implies the existence of a finite Lipschitz constant $0<L_\kappa<\infty$ in the second line of \eqref{eq:bound_Delta_u}. Similar to \eqref{eq:bound_init_hat}, the norm $\| \hat{\vm x}\inds - \hat{\vm x}\indqp \|$ is bounded by \begin{align} \label{eq:bound_opt_hat}
		\| \hat{\vm x}\inds - \hat{\vm x}\indqp \| \leq \sqrt{p}\| \vm x\inds- \vm x\indqp\|\,.
	\end{align}
	Inserting \eqref{eq:bound_opt_hat} and \eqref{eq:bound_Delta_u} into \eqref{eq:bound_Delta_x}, leads to
	\begin{align}
		\|\Delta \vm x\indq(\tau; \initstate)\| &\leq L_f\int_0^\tau (1+L_\kappa) \|\Delta \vm x \indq\|  + (\sqrt{p} + L_\kappa) \|\Delta \vm x\indqp\| + L_\kappa \| \Delta \vm \lambda\indqp\|\, \dd s\,.
	\end{align}
	Using Gronwall's inequality and taking the $L_\infty$-norm on both sides, one obtains
	\begin{align} \label{eq:conv_x_short}
		\|\Delta \vm x\indq(\cdot; \initstate)\|_{L_\infty} \leq C_1 \| \Delta \vm x\indqp\|_{L_\infty} + C_2 \| \Delta \vm \lambda\indqp\|_{L_\infty}
	\end{align}
	with $C_1:= L_fT(\sqrt{p} + L_\kappa)\mathrm{e}^{L_f(1+L_\kappa)T}>0$ and $C_2:= L_fL_\kappa T\mathrm{e}^{L_f(1+L_\kappa)T}>0$. 
	
	Next, a bound for $\Delta \vm \lambda\indq(\tau; \initstate)$ needs to be found. To this end, consider the optimal adjoint dynamics \eqref{eq:opt_adjointdyn} of the central OCP \eqref{eq:central_ocp} and the adjoint dynamics \eqref{eq:iter_adjoint_dyn} of the local OCP \eqref{eq:local_ocp}. The functions $\vm G$ and $ \vm G_d$ are structurally equivalent as the sensitivities extend the local cost functions \eqref{eq:local_costFunction} such that the equations \eqref{eq:opt_adjointdyn} and \eqref{eq:iter_adjoint_dyn} involve the same terms.\cite{Pierer}
	Consider the integral form of \eqref{eq:opt_adjointdyn} and \eqref{eq:iter_adjoint_dyn} in reverse time and the notation $y_r(\tau) := y(T - \tau) $ for some trajectory $y(\tau),\, \timetau$
	\begin{align} 
		\| \Delta \vm{\lambda}_r\indq(\tau; \initstate)\| &\leq L_V \|\Delta \vm{x}_r\indq(0)\| + \int_0^\tau \| \vm G_d(\vm{x}_r\inds, \vm{u}_r\inds, \vm{\lambda}_r\inds; \vm{\bar{x}}_r\inds, \vm{\bar{\lambda}}_r\inds) -\vm G_d(\vm x\indq_r, \vm u\indq_r, \vm \lambda\indq_r; \vm{\bar x}_r\indqp,\vm{\bar \lambda}_r\indqp) \|\, \dd s \nonumber \\
		&\leq L_V\|\Delta \vm x_r\indq(0)\| + L_G \int_0^\tau \|\Delta \vm x_r\indq \| + \| \Delta \vm u_r\indq \| + \| \Delta \vm \lambda_r\indq\|+ \| \bar{\vm x}_r\inds - \bar{\vm x}_r\indqp \| + \| \bar{\vm \lambda}_r\inds - \bar{\vm \lambda}_r\indqp \| \, \dd s  \label{eq:bound_Delta_lambda}\\
		&\leq L_V\|\Delta \vm x_r\indq(0)\| + L_G \int_0^\tau (1 +L_\kappa) \|\Delta \vm x_r\indq \| + \|\Delta \vm \lambda_r \indq\|+ (\sqrt{p_x}+L_\kappa)\| \Delta \vm x_r\indqp \| +(\sqrt{p_\lambda}+L_\kappa) \| \Delta \vm \lambda_r\indqp\|) \, \dd s \nonumber
	\end{align}
	where $\vm G(\vm x\inds, \vm u\inds, \vm \lambda\inds) = \vm G_d(\vm x\inds, \vm u\inds, \vm \lambda\inds; \vm{\bar x}\inds,\vm{\bar \lambda}\inds)$ is used. 
	
	By assumption any solution of the adjoint dynamics \eqref{eq:opt_adjointdyn} yields a bounded solution for any continuous and bounded state trajectory $\vm x(\tau) \in \mathbb{X}$ and input $\vm u(\tau) \in \mathbb{U}$, $\timetau$, i.e. there exist a compact set $\mathbb{X}_\lambda \subset \mathds{R}^n$ such that $\vm \lambda(\tau) \in \mathbb{X}_\lambda$, for all $\timetau$. Thus, by choosing $T<\min\{T_x,T_\lambda\}$ all trajectories in \eqref{eq:bound_Delta_lambda} are defined within compact sets such that the Lipschitz estimate \eqref{eq:LipschitzG} is applicable. By applying Gronwall's inequality and the $L_\infty$-norm, the bound 
	\begin{equation} \label{eq:conv_lambda_short1}
		\|\Delta \vm \lambda\indq(\cdot; \initstate)\|_{L_\infty} \leq C_3 \| \Delta \vm x\indq\|_{L_\infty} + C_4\| \Delta \vm  x \indqp\|_{L_\infty} + C_5\| \Delta \vm \lambda \indqp\|_{L_\infty}
	\end{equation}
	with $C_3:= (L_V + TL_G(1+L_\kappa))\mathrm{e}^{L_G T}$ and $C_4:=TL_G(\sqrt{ p_x}+L_\kappa)\mathrm{e}^{L_G T}$, $C_5:=TL_G(\sqrt{p_{\lambda}}+L_\kappa)\mathrm{e}^{L_G T}$ is obtained. 
	Inserting \eqref{eq:conv_x_short} into \eqref{eq:conv_lambda_short1} leads to 
	\begin{equation}\label{eq:conv_lambda_short2}
		\|\Delta \vm \lambda\indq(\cdot; \initstate)\|_{L_\infty} \leq (C_3C_1 + C_4)\| \Delta \vm x\indqp\|_{L_\infty}+ (C_3 C_2 + C_5)\| \Delta \vm \lambda \indqp\|_{L_\infty}.
	\end{equation}
	Concatenating \eqref{eq:conv_x_short} and \eqref{eq:conv_lambda_short2} results in the linear discrete-time system
	\begin{equation}
		\begin{bmatrix}
			\|\Delta \vm x\indq(\cdot; \initstate)\|_{L_\infty} \\
			\|\Delta \vm \lambda\indq(\cdot; \initstate)\|_{L_\infty} \\
		\end{bmatrix} \leq \begin{bmatrix}
			C_1 & C_2 \\
			C_3C_1 + C_4 & C_3C_2 + C_5
		\end{bmatrix} \begin{bmatrix}
			\|\Delta \vm x\indqp(\cdot; \initstate)\|_{L_\infty} \\
			\|\Delta \vm \lambda\indqp(\cdot; \initstate)\|_{L_\infty}  \\
		\end{bmatrix} = \vm C \begin{bmatrix}
			\|\Delta \vm x\indqp(\cdot; \initstate)\|_{L_\infty} \\
			\|\Delta \vm \lambda\indqp(\cdot; \initstate)\|_{L_\infty}  \\
		\end{bmatrix} \,,
	\end{equation}
	where the inequality is to be understood element-wise and $\vm C\in \mathds{R}^{2\times2}$ denotes the iteration matrix. Finally, Taking the Euclidean norm $\| \cdot\|$ on both sides proves \eqref{eq:lin_conv} with $p=\|\vm C\|$. Note that $ p\rightarrow 0 $ for $T \rightarrow 0$. Hence, there exists a maximum horizon length $T_p$ such that for all $T<T_p$ the contraction ratio $p$ satisfies $p<1$. Consequently, for $T< T_{\mathrm{max}}:= \min\{T_p, T_x, T_\lambda\}$ the iterates are bounded and the algorithm converges. 
\end{proof}

\bmsection{Proof of Lemma 2}
\begin{proof}
	To prove Lemma \ref{lem:error_bound}, the error norm $\|\Delta \vm x_c\indqk(\tau; \initstate)\|$, $\timetausample$ is expanded such that
	\begin{align} \label{eq:norm_expansion}
		\|\Delta \vm x_c\indqk(\tau; \initstate)\|&= \|\vm x_c\indqk(\tau; \initstate) -\vm x\indqk(\tau; \initstate) + \vm x\indqk(\tau; \initstate)  - \vm x\inds(\tau; \initstate)\|\leq \|\vm x_c\indqk(\tau; \initstate) -\vm x\indqk(\tau; \initstate)\| + \|\vm x\indqk(\tau; \initstate)  - \vm x\inds(\tau; \initstate)\|\nonumber\\
		&=\| \vm \xi\indqk(\tau;\initstate)\| + \|\Delta \vm x\indqk(\tau; \initstate)\|
	\end{align}
	with the error between predicted and actual trajectories $\vm \xi\indqk(\tau;\initstate):= \vm x_c\indqk(\tau; \initstate) -\vm x\indqk(\tau; \initstate)$ as well as predicted and optimal trajectories $\Delta \vm x\indqk(\tau; \initstate):= \vm x\inds(\tau; \initstate) -\vm x\indqk(\tau; \initstate)$.
	We derive a bound on $\vm \xi\indqk(\tau;\initstate),\, \timetausample$ by considering the integral form of the dynamics \eqref{eq:pred_traj} and \eqref{eq:act_traj} as well as the Lipschitz property of the control in Assumption~\ref{ass:cost_Lipschitz}
	\begin{align}\label{eq:bound_xi_1}
		\| \vm \xi\indqk(\tau;\initstate)\| &\leq  \int_0^\tau  \|\vm{ \hat f}(\vm x_c\indqk, \vm u\indqk, \vm {\hat x}_c\indqk )-\vm{ \hat f}(\vm x\indqk, \vm u\indqk, \vm {\hat x}\indqkp)\|\, \dd s \leq L_f \int_0^\tau \| \vm x_c\indqk - \vm x\indqk\| + \| \vm {\hat x}_c\indqk - \vm {\hat x}\indqkp\|\, \dd s
	\end{align}
	where again $\vm{\hat x} = [\Ni{x}]_{\agents} \in \mathds{R}^{p}$ is used to achieve structural equivalence between $\vm f(\vm  x_c\indqk, \vm u \indqk)$ and $\vm{ \hat f}(\vm x\indqk, \vm u\indqk, \vm {\hat x}\indqkp)$. Note that any solution of the differentiable dynamics \eqref{eq:centr_dyn} is bounded for bounded controls, i.e. $ \vm x\inds(t),\, \vm x_c(t) \in \mathbb{X}$ for all $ \vm u(t) \in \mathbb{U}$ and $\initstate \in \Gamma_{\alpha}$. In addition, Lemma \ref{lem:lin_conv} implies the boundedness of any solution $\vm x\indqk(t)$ of \eqref{eq:pred_traj} at a given iteration $q$. Together this implies the existence of a finite Lipschitz constant $L_f>0$ in \eqref{eq:bound_xi_1}.
	Similar to \eqref{eq:bound_init_hat},  the norm $\| \vm {\hat x}_c\indqk - \vm {\hat x}\indqkp\|$ in \eqref{eq:bound_xi_1} is bounded by 
	\begin{align}\label{eq:bound_hat}
		\| \vm {\hat x}_c\indqk - \vm {\hat x}\indqkp\| \leq \sqrt{ p} \| \vm {\hat x}_c\indqk - \vm {\hat x}\indqkp\|_\infty = \sqrt{ p} \| \vm {x}_c\indqk - \vm {x}\indqkp\|_\infty \leq \sqrt{ p} \| \vm {x}_c\indqk - \vm {x}\indqkp\|\,. 
	\end{align}
	This norm can further be expanded to
	\begin{align}\label{eq:bound_hat_2}
		\| \vm {x}_c\indqk - \vm {x}\indqkp\| = \| \vm {x}_c\indqk -\vm {x}\indqk + \vm {x}\indqk- \vm {x}\indqkp\| \leq \| \vm {x}_c\indqk -\vm {x}\indqk\| + \|\vm {x}\indqk- \vm {x}\indqkp\|
	\end{align}
	such that \eqref{eq:bound_hat} is compatible with \eqref{eq:bound_xi_1}. Inserting \eqref{eq:bound_hat_2} into \eqref{eq:bound_xi_1}, results in
	\begin{align}
		\| \vm \xi\indqk(\tau;\initstate)\| &\leq L_f  \int_0^\tau  (1+ \sqrt{p})\| \vm \xi\indqk\| +  \sqrt{p} \|\vm {x}\indqk- \vm {x}\indqkp\|\,\dd s.
	\end{align}
	Applying Gronwall's inequality and the $L_\infty$ norm, eventually leads to
	\begin{align}\label{eq:bound_xi_2}
		\| \vm \xi\indqk(\cdot; \initstate)\|_{L_\infty} \leq  L_f\sqrt{p}\Delta t\mathrm{e}^{L_f(1+\sqrt{ p})\Delta t}\|\vm {x}\indqk(\cdot; \initstate)- \vm {x}\indqkp(\cdot; \initstate)\|_{L_\infty} \leq c_3 d  \| \initstate\|
	\end{align}
	with the constant $c_3:= N L_f\sqrt{p}\Delta t\mathrm{e}^{L_f(1+\sqrt{ p})\Delta t} >0$ where the stopping criterion \eqref{eq:stopp_crit} was used to bound the norm between two state iterates in \eqref{eq:bound_xi_2}. In what follows, a bound on the remaining unknown norm $ \|\Delta \vm x\indqk(\tau; \initstate)\|$ is derived in similar fashion by considering the integral form of the dynamics \eqref{eq:pred_traj} and \eqref{eq:opt_traj} for $\timetausample$
	\begin{align}\label{eq:bound_delta_x_1}
		\|\Delta \vm x\indqk(\tau; \initstate)\| &\leq \int_0^\tau \| \vm{ \hat f}(\vm x\inds, \vm u\inds, \vm{\hat{x}}\inds) - \vm{ \hat f}(\vm x\indqk, \vm u\indqk, \vm{\hat{x}}\indqkp)\|\, \dd s \leq L_f  \int_0^\tau \|\vm x\inds - \vm x\indqk\| + \|\vm u\inds - \vm u \indqk\| + \|\vm{\hat{x}}\inds - \vm{\hat{x}}\indqkp\|\, \dd s
	\end{align}
	with the unknown norms $ \|\vm u\inds - \vm u \indqk\|$ and $ \|\vm{\hat{x}}\inds - \vm{\hat{x}}\indqkp\|$. Similar to \eqref{eq:bound_hat}, the norm $\|\vm{\hat{x}}\inds - \vm{\hat{x}}\indqkp\|$ is bounded by
	\begin{align}\label{eq:bound_hat_3}
		\|\vm{\hat{x}}\inds - \vm{\hat{x}}\indqkp\| \leq  \sqrt{p}(\|\vm x\inds - \vm x\indqk\| +  \|\vm {x}\indqk- \vm {x}\indqkp\|). 
	\end{align}
	which can be inserted into \eqref{eq:bound_delta_x_1}
	\begin{align} \label{eq:bound_delta_x_3}
		\|\Delta \vm x\indqk(\tau; \initstate)\| &\leq   \int_0^\tau L_f (1+ \sqrt{p}) \|\Delta \vm x\indqk\|  +L_f\|\vm u\inds - \vm u \indqk\|+ L_f\sqrt{p}  \|\vm x\indqk - \vm x\indqkp\|\, \dd s \nonumber \\
		&\leq N L_f \sqrt{p} \tau  d\| \initstate\| +  \int_0^\tau L_f (1+ \sqrt{p}) \|\Delta \vm x\indqk\|  +L_f\|\vm u\inds - \vm u \indqk\|\, \dd s
	\end{align}
	where again the stopping criterion \eqref{eq:stopp_crit} was used to bound $\|\vm x\indqk - \vm x\indqkp\|$ in \eqref{eq:bound_delta_x_3}.
	The difference between optimal $\vm u^*$ and suboptimal input $\vm u\indqk$ can be bounded as follows 
	\begin{align}\label{eq:bound_optimal_input}
		&\int_0^\tau \| \vm u\inds(s;\initstate) - \vm u\indqk(s;\initstate) \|\, \dd s = \int_0^\tau\|\vm \kappa\inds(\vm x\inds(s; \initstate); \vm x\inds,\vm \lambda\inds, \initstate ) - \vm \kappa(\vm x\indqk(s; \initstate); \vm x\indqkp,\vm \lambda\indqkp, \initstate )\|\, \dd s \nonumber\\
		&\leq L_\kappa \int_0^\tau \|\vm x\inds - \vm x\indqk\| +  \|\vm x\inds - \vm x\indqkp\| + \|\vm \lambda\inds - \vm \lambda\indqkp\|\, \dd s\nonumber \leq L_\kappa \int_0^\tau \|\vm x\inds - \vm x\indqk\|\, \dd s+ \sqrt{2}L_\kappa \tau \bigg\| \begin{matrix}
			\vm x\inds - \vm x\indqkp\\ \vm \lambda\inds - \vm \lambda\indqkp
		\end{matrix} \bigg \|_{L_\infty} \nonumber  \\
		&\leq L_\kappa \int_0^\tau \|\vm x\inds - \vm x\indqk\|\, \dd s+ \frac{\sqrt{2}L_\kappa \tau}{1-p}\bigg\| \begin{matrix}
			\vm x\indqk - \vm x\indqkp\\ \vm \lambda\indqk - \vm \lambda\indqkp
		\end{matrix} \bigg \|_{L_\infty} \leq L_\kappa \int_0^\tau \|\Delta \vm x\indqk\| \, \dd s + \frac{\sqrt{2}L_\kappa \tau}{1-p} N d \| \initstate\|
	\end{align}
	where $  \vm \kappa\inds(\vm x\inds(\tau, \initstate); \initstate) = \vm \kappa(\vm x\inds(\tau; \initstate); \vm x\inds,\vm \lambda\inds, \initstate ) $ is used. Note that  the states $ \vm x\inds$, $\vm x\indqk$ and $\vm x\indqkp$ are defined within a compact set. The same holds for $\vm \lambda\inds$ and $ \vm \lambda\indqkp$. Consequently, Assumption \ref{ass:cost_Lipschitz} implies that there exists a Lipschitz constant $L_{\kappa}<\infty$ such that first inequality in \eqref{eq:bound_optimal_input} holds. Furthermore, the linear convergence property \eqref{eq:lin_conv} in Lemma \ref{lem:lin_conv} in combination with the stopping criterion \eqref{eq:stopp_crit} was used in the last two inequalities in \eqref{eq:bound_optimal_input}. 
	Combining \eqref{eq:bound_optimal_input} and \eqref{eq:bound_hat_3} with \eqref{eq:bound_delta_x_1} leads to
	\begin{align}
		\|\Delta \vm x\indqk(\tau; \initstate)\| &\leq \int_0^\tau c_4\|\Delta \vm x\indqk\|\, \dd s +c_5 \tau d \| \initstate \|
	\end{align}
	with constants $c_4 := L_f (1 + \sqrt{p} + L_\kappa)>0$ and $c_5:=N L_f \big(\sqrt{p} + \frac{\sqrt{2} L_\kappa}{1 - p}\big)>0$. Applying Gronwall's inequality and the $L_\infty$ eventually leads to
	\begin{equation} \label{eq:bound_delta_x_2}
		\|\Delta \vm x\indqk(\cdot; \initstate)\|_{L_\infty} \leq c_5 \Delta t \mathrm{e}^{c_4 \Delta t} d \| \initstate\|\,.
	\end{equation}
	Finally, the bound on the error \eqref{eq:norm_expansion} is given by \eqref{eq:bound_xi_2} and \eqref{eq:bound_delta_x_2}
	\begin{equation}
		\|\Delta \vm x_c\indqk(\cdot; \initstate)\|_{L_\infty} \leq D d \| \initstate \| 
	\end{equation}
	with $D:= c_3 + c_5 \Delta t \mathrm{e}^{c_4 \Delta t}$ which proves the lemma. 
\end{proof}
\bmsection{Proof of Lemma 3}
The proof follows along the lines of \cite{Graichen,Graichen2,Bestler} and considers the bound on the actual state trajectory in the next sampling step $\Delta t$
\begin{align} \label{eq:bound_lemma3}
	\|  \vm x_c\indqk(\Delta t; \initstate)\| \leq \| \vm x\inds(\Delta t; \initstate)\| + \|\Delta \vm x\indqk_c(\Delta t; \initstate)\| \leq  (\mathrm{e}^{\hat L \Delta t } + D d ) \|\initstate\|\,, 
\end{align}
where the second line follows from \eqref{eq:lemma_2} in Lemma \ref{lem:error_bound} and the bound on the optimal state trajectory in \eqref{eq:state_upper_bound}. Moreover, note that, for any $ \vm x \in \Gamma_\alpha$ with $\alpha>0$ the bound $ \| \vm x\|< \sqrt{\frac{\alpha}{m_J}}$ follows from the set definition \eqref{eq:reg_of_attr} and the lower bound on the optimal cost \eqref{eq:lower_bound_opt_cost}. Vice versa, $\| \vm x\|< \sqrt{\frac{\alpha}{M_J}}$ implies by the upper bound \eqref{eq:upper_bound_opt_cost} that $ \vm x \in \Gamma_\alpha$. Thus, $\initstate \in \Gamma_{\bar \alpha}$ with $ \bar \alpha:= \frac{m_J \alpha}{M_J (\mathrm{e}^{\hat L \Delta t } + D d )^2}$ has the consequence that $ \| \initstate\| \leq \frac{1}{\mathrm{e}^{\hat L \Delta t } + D d} \sqrt{ \frac{\alpha}{M_J}}$ which inserted in \eqref{eq:bound_lemma3} leads to
\begin{equation}
	\|  \vm x_c\indqk(\Delta t; \initstate)\| \leq \sqrt{ \frac{\alpha}{M_J}}.
\end{equation}
This implies that $\vm x_c\indqk(\Delta t; \initstate) \in \Gamma_\alpha$ for $ \initstate \in \Gamma_{\bar \alpha}$. Now, we can estimate the difference between the optimal cost at point $ \vm {\munderbar x}_{k+1} = \vm x_c\indqk(\Delta t, \vm {\munderbar x}_k)$ and at $ \vm {\munderbar x}_{k+1}\inds = \vm x_c\inds(\Delta t; \initstate)$ (which both lie in $\Gamma_{\alpha}$) by considering the following line integral along the linear path $\vm {\munderbar x}_{k+1}\inds + s \Delta \vm {\munderbar x}_{k+1}$ with $\Delta \vm {\munderbar x}_{k+1}:= \Delta \vm x_c\indqk(\Delta t; \initstate)$ and the path parameter $s \in [0,1]$
\begin{align}
	J\inds(\vm {\munderbar x}_{k+1}) &= J\inds(\vm {\munderbar x}_{k+1}\inds) + \int_0^1 \nabla J\inds(\vm {\munderbar x}_{k+1}\inds + s \Delta \vm {\munderbar x}_{k+1} ) \Delta \vm {\munderbar x}_{k+1}\, \dd s \nonumber\\
	&= J\inds(\vm {\munderbar x}_{k+1}\inds) + \int_0^1 \nabla J\inds(\vm {\munderbar x}_{k+1}\inds) + \int_0^s \bigg[ \nabla^2 J( \vm {\munderbar x}_{k+1}\inds + s_2 \Delta \vm {\munderbar x}_{k+1} )\Delta \vm {\munderbar x}_{k+1}\, \dd s_2 \bigg] \Delta\vm {\munderbar x}_{k+1} \, \dd s \nonumber\\
	&\leq J\inds(\vm {\munderbar x}_{k+1}\inds ) + B \| \vm {\munderbar x}_{k+1}\inds \| \|\Delta \vm x_c\indqk(\Delta t; \initstate)\| + \frac{1}{2} B \|\Delta \vm x_c\indqk(\Delta t; \initstate)\|^2.
\end{align}
The last line follows from the twice differentiability of the optimal cost $J\inds(\cdot)$, see Assumption \ref{ass:cost_Lipschitz}, which implies that there exists a constant $B>0$ such that $\|\nabla J\inds(\vm x)\| \leq B \| \vm x\|$ and $\|\nabla^2 J\inds(\vm x)\|\leq B$ for all $ \vm x \in \Gamma_\alpha$. Note that by definition of $\bar{\alpha}$ the linear path lies completely within $\Gamma_\alpha$. Considering the optimal MPC case \eqref{eq:th1}, the first term $J\inds(\vm {\munderbar x}_{k+1}\inds )$ can be related to the previous optimal cost by \eqref{eq:bound_opt_integral} which results in the contraction term $J\inds(\vm {\munderbar x}_{k+1}\inds ) \leq (1-a) J\inds(\initstate\inds)$. The bounds \eqref{eq:state_upper_bound} and \eqref{eq:lower_bound_opt_cost}  give $ \| \vm {\munderbar x}_{k+1}\inds \| \leq \frac{\mathrm{e}^{\hat L \Delta t}}{\sqrt(m_J)} \sqrt{J\inds(\initstate)}$, which yields \eqref{eq:lemma_2} with $0\leq a<1$ as in \eqref{eq:bound_opt_integral}, $b:=  \frac{\mathrm{e}^{\hat L \Delta t}}{\sqrt{m_J}}$ and $c=\frac{B}{2}$. This completes the proof of Lemma~\ref{lem:cost_bound}.
\bmsection{Additional bounds}
This section states some useful bounds on the optimal state trajectory $\vm x\inds(\tau, \initstate)$, adjoint trajectory $\vm \lambda\inds(\tau; \initstate)$ and optimal cost $J\inds(\initstate)$. Similar bounds are given in the references\cite{Graichen,Graichen2,Bestler}. Note that the optimal state trajectory lies in the compact set $\Gamma_\alpha$, i.e. $ \vm x\inds(\tau; \initstate) \in \Gamma_{\alpha}$ for $\timetau$ and $ \initstate \in \Gamma_\alpha$ which follows from~\eqref{eq:th1}. Considering the equilibrium $ \vm f(\vm 0,\vm 0) = \vm 0$ as well as $\vm \kappa\inds(\vm 0; \initstate)= \vm 0$ for the optimal feedback \eqref{eq:MPC_control_law}, the following Lipschitz estimates hold $ \| \vm f(\vm x, \vm u)\| = \|[\vm f_i(\vm x_i, \vm u_i, \Ni{x})]_{\agents}\| \leq L_f (\|\vm x\| + \|\vm u\|)$ and $ \vm \kappa\inds(\vm x; \initstate) \leq L_{\kappa\inds} \| \vm x\|$ for all $\vm x(t) \in \mathbb{X}$ and $\vm u(t) \in \mathbb{U}$ with $L_f$, $L_{\kappa\inds}<\infty$. An upper bound on the optimal state trajectory $\vm x\inds(\tau, \initstate)$ with $\vm x\inds(0) = \initstate$ can be obtained based on Assumption \ref{ass:cost_Lipschitz} and Gronwall's inequality
\begin{align} \label{eq:state_upper_bound}
	\|\vm x\inds(\tau; \initstate)\| &                                                                                                                                                                                                                                                                                                                                                                                                                                                                                                                                                               \leq \|\initstate\| + \int_0^\tau \|\vm f(\vm x\inds(s; \initstate), \vm u\inds(s; \initstate))\|\, \dd s \leq \|\initstate\| + \int_0^\tau \|\vm f(\vm x\inds(s; \initstate), \vm \kappa\inds(\vm x\inds(s; \initstate); \initstate))\|\, \dd s \\
	&\leq \|\initstate\| +L_f \int_0^\tau \|\vm x\inds(s;\initstate)\| + \|\vm \kappa\inds(\vm x\inds(s; \initstate); \initstate))\|\, \dd s \leq \|\initstate\|+  L_f(1 + L_{\kappa\inds}) \int_0^\tau \|\vm x\inds(s;\initstate)\|\, \dd s  \leq\|\initstate\| \mathrm{e}^{\hat L \tau } \nonumber
\end{align}
with $\hat L :=  L_f(1 + L_{\kappa\inds})$. Similarly, a lower bound can be found by the inverse Gronwall inequality
\begin{align} \label{eq:state_lower_bound}
	\|\vm x\inds(\tau; \initstate)\| &\geq \|\initstate\| - \int_0^\tau \|\vm f(\vm x\inds(s; \initstate), \vm u\inds(s; \initstate))\|\, \dd s \geq\|\initstate\| \mathrm{e}^{-\hat L \tau }\,.
\end{align}
The difference between two optimal state trajectories $\vm x\inds(\tau; \vm {\munderbar x}_{k+1})$ and $\vm x\inds(\tau; \vm {\munderbar x}_k)$ can be estimated as follows  
\begin{align}
	\|\vm x\inds(\tau; \vm {\munderbar x}_{k+1}) -\vm x\inds(\tau; \vm {\munderbar x}_k) \| &\leq \| \vm {\munderbar x}_{k+1} - \initstate \| + L_f \int_0^\tau \|\vm x\inds(s; \vm {\munderbar x}_{k+1}) -\vm x\inds(s; \vm {\munderbar x}_k)\| + \| \vm \kappa\inds(\vm x\inds(s; \vm {\munderbar x}_{k+1}); \vm {\munderbar x}_{k+1}) - \vm \kappa\inds(\vm x\inds(s; \vm {\munderbar x}_k); \vm {\munderbar x}_k ) \|\, \dd s \nonumber \\
	&\leq \bar L\|\vm {\munderbar x}_{k+1} - \initstate\| + \hat L \int_0^\tau \|\vm x\inds(s; \vm {\munderbar x}_{k+1}) -\vm x\inds(s; \vm {\munderbar x}_k)\|\, \dd s
\end{align}
with $\bar L : = (1 + L_f L_{k\inds})$. Applying Gronwall's inequality and the $L_\infty$-norm leads to 
\begin{equation} \label{eq:optstate_bound}
	\|\vm x\inds(\cdot; \vm {\munderbar x}_{k+1}) -\vm x\inds(\cdot; \vm {\munderbar x}_k) \|_{L_\infty} \leq \bar L \mathrm{e}^{\hat L T} \|\vm {\munderbar x}_{k+1} - \initstate\| \leq L_x  \|\vm {\munderbar x}_{k+1} - \initstate\|
\end{equation}
with $L_x :=\bar L \mathrm{e}^{\hat L T}$. 
Similar to \eqref{eq:optstate_bound}, a bound between two optimal adjoint state trajectories $\vm \lambda \inds(\tau; \vm {\munderbar x}_{k+1})$ and $\vm \lambda\inds(\tau; \vm {\munderbar x}_k)$ can be found by integration of the adjoint dynamics \eqref{eq:opt_adjointdyn} in reverse time with the notation $y_r(\tau) := y(T - \tau) $ for some trajectory $y(\tau),\, \timetau$, i.e.
\begin{align} \label{eq:bound_opt_lambda}
	&\|\vm \lambda_r\inds(\tau; \vm {\munderbar x}_{k+1}) -\vm \lambda_r\inds(\tau; \vm {\munderbar x}_k) \| \leq L_V \|\vm{x}_r\indq(0;\vm{\munderbar x}_{k+1}) - \vm{x}_r\indq(0;\vm{\munderbar x}_{k})\| \nonumber \\
	&+ L_G\int_0^\tau \|\vm x_r\inds(s; \vm{\munderbar x}_{k+1}) -\vm x_r\inds(s; \vm{\munderbar x}_k) \| + \| \vm \kappa\inds(\vm x_r\inds(s; \vm {\munderbar x}_{k+1}); \vm {\munderbar x}_{k+1}) - \vm \kappa\inds(\vm x_r\inds(s; \vm {\munderbar x}_k); \vm {\munderbar x}_k ) \| +  \|\vm \lambda_r\inds(s; \vm{\munderbar x}_{k+1})) - \vm \lambda_r\inds(s; \vm{\munderbar x}_k))\|\, \dd s \nonumber \\
	&\leq L_G L_{\kappa\inds} \| \vm x_{k+1} - \vm x_k \| + (L_V+TL_G(1+L_{\kappa\inds}))	\|\vm x\inds(\cdot; \vm {\munderbar x}_{k+1}) -\vm x\inds(\cdot; \vm {\munderbar x}_k) \|_{L_\infty} + L_G \int_0^\tau \| \vm \lambda\inds(s; \vm{\munderbar x}_{k+1})) - \vm \lambda\inds(s; \vm{\munderbar x}_k))\|\, \dd s \nonumber\\
	&\leq (L_G L_{\kappa\inds} + L_x(L_V+TL_G(1+L_{\kappa\inds}))\mathrm{e}^{L_GT}\| \vm {\munderbar x}_{k+1} - \initstate\| \leq L_{\lambda}\| \vm {\munderbar x}_{k+1} - \initstate\|
\end{align}
with $L_{\lambda}:=(L_G L_{\kappa\inds} + L_x(L_V+TL_G(1+L_{\kappa\inds}))\mathrm{e}^{L_GT}$. Note that the optimal adjoint states $\vm \lambda\inds$ are defined on the compact set $\mathbb{X}_\lambda$ such that the finite Lipschitz constant $L_G$ in the second line of \eqref{eq:bound_opt_lambda} exists.

Utilizing \eqref{eq:state_upper_bound} and \eqref{eq:state_lower_bound}, an upper bound on the optimal cost can be found by using the bounds for the terminal and integral costs \eqref{eq:bound_costs}
\begin{align}\label{eq:upper_bound_opt_cost}
	J\inds(\initstate) &\leq M_V \|\vm x\inds(T;\initstate)\|^2 + M_l \int_0^T \|\vm x \inds(\tau; \initstate)\|^2 + \|\vm u\inds(\tau; \initstate)\|^2\, \dd \tau \leq M_V \mathrm{e}^{2\hat L T} \|\initstate\| + M_l \| \initstate\|^2 \int_0^T \mathrm{e}^{2\hat L \tau } + L_{k\inds}^2 \mathrm{e}^{2 \hat L T}\, \dd \tau \nonumber \\
	&= M_J \|\initstate\|^2
\end{align}
with $M_J:= M_V \mathrm{e}^{2\hat L T} + \frac{M_l}{2 \hat L} (\mathrm{e}^{2\hat L T}(1 + L_{k\inds}^2) - 1 -L_{k\inds}^2) $ as well as a lower bound 
\begin{align} \label{eq:lower_bound_opt_cost}
	J^*(\initstate) &\geq m_l \int_0^T \|\vm x\inds(\tau; \initstate)\|^2\, \dd \tau \geq m_l \|\initstate\|^2 \int_0^T \mathrm{e}^{-2\hat L \tau } =m_J \|\initstate\|^2 
\end{align}
with $m_J:= \frac{m_l}{2\hat L}(1- \mathrm{e}^{-2\hat L T})$. Additionally, the integral in \eqref{eq:th1} can be lower bounded by 
\begin{align} \label{eq:bound_opt_integral}
	\int_0^{\Delta t} l(\vm x\inds(\tau; \initstate), \vm u\inds(\tau; \initstate)\,\dd \tau &\geq m_l \int_0^{\Delta t} \|\vm x\inds(\tau; \initstate)\|^2\,\dd \tau \geq a J\inds(\initstate)
\end{align}
with $a:= \frac{m_l(1-\mathrm{e}^{-2\hat L \Delta t})}{2 \hat L M_J}$.
\end{document}